\documentclass[12pt, reqno]{amsart}
\usepackage{psfrag}
\usepackage{amsmath,amssymb,amsfonts,amstext, amsthm, amscd}
\usepackage{verbatim}
\usepackage{graphicx}
\usepackage{enumerate}
\newtheorem{theorem}{Theorem}
\newtheorem{lemma}[theorem]{Lemma}
\newtheorem{proposition}[theorem]{Proposition}
\newtheorem{corollary}[theorem]{Corollary}

\newtheorem{conjecture}{Conjecture}

\newtheorem*{claim*}{Claim}

\theoremstyle{definition}
\newtheorem*{definition*}{Definition}

\theoremstyle{remark}

\newcommand{\df}{\bf \em}

\newcommand\ns[1]{ \left\{ {#1} \right\} }
\newcommand\M{{\mathcal M}}        % some sigma field
\DeclareMathOperator{\metric}{d}
\DeclareMathOperator{\metricstar}{d^{*}_{\mathrm{Lip}}}
\DeclareMathOperator{\metricbar}{\bar{d}}
\DeclareMathOperator*{\argmin}{argmin}

\newcommand{\Z}{{\mathbb Z}}
\newcommand{\R}{{\mathbb R}}
\newcommand{\N}{{\mathbb N}}
\newcommand{\T}{{\mathbb T}}
\newcommand{\e}{\varepsilon}
\newcommand{\follow}{\mathbf{D}}
\newcommand{\nullcon}{{\mathbf{V}}}
\newcommand{\avg}{\mathrm{A}}
\newcommand{\boys}{\mathsf{B}}
\newcommand{\girls}{\mathsf{G}}
\newcommand{\es}{\mathsf{EG}}
\newcommand{\badset}{\mathsf{BS}}

\newcommand{\Bl}{\textrm{[BL]}}

\newcommand\fracc[2]{ #1 /#2 }      %alternate form of \frac
\newcommand\X{{\Omega}}            % State space
\newcommand\garbage[1]{}

\newcommand{\poly}{\mathrm{P}}
\newcommand{\q}{\mathcal{Q}_S}

\begin{document}

\title[Ergodic universality]{Ergodic universality of some topological dynamical systems}
\author[A. Quas]{Anthony Quas}
\author[T. Soo]{Terry Soo}

\address[A. Quas]{Department of Mathematics and Statistics, University
  of Victoria,
PO BOX 3060 STN CSC, Victoria, BC V8W 3R4, Canada}
\email{aquas(at)uvic.ca}  \urladdr{http://www.math.uvic.ca/$\sim$aquas/}

\address[T.  Soo]{Department of Statistics, University of Warwick, Coventry,
CV4 7AL, United Kingdom}

\email{t.soo(at)warwick.ac.uk}
\urladdr{www.warwick.ac.uk/fac/sci/statistics/staff/academic-research/soo}

\thanks{Funded in part by NSERC and MSRI (both authors)}

\begin{abstract}
  The Krieger generator theorem says that every invertible ergodic measure-preserving system
  with finite measure-theor\-etic entropy can be embedded into a full shift with strictly
  greater topological entropy.  We extend Krieger's theorem to include
   toral automorphisms, and more generally, any topological
  dynamical system on a compact metric space that satisfies almost weak specification,
  asymptotic entropy expansiveness, and the small boundary property. As a corollary, one
  obtains a complete solution to a natural generalization of an open problem in Halmos's 1956
  book regarding an isomorphism invariant that he proposed.
  \end{abstract}

\maketitle

\section{Introduction}
\label{intro}

Let $S$ be a self-homeomorphism of a compact metric space $(Y,\metric)$.
Let $T$ be an invertible ergodic measure-preserving transformation on a non-atomic  (Lebesgue) probability space $(\X, \mu)$.
An {\df embedding} of $(\X, \mu, T)$ into $(Y, S)$ is a measurable mapping
$\Psi: \X \to Y$  such that  the restriction of $\Psi$ to a
set of full measure $\X'$ is an injection, and $\Psi( T(\omega)) =
S(\Psi(\omega))$ for all $\omega \in \X'$.    We say that the topological dynamical system $(Y,S)$ is {\df universal}
if for every invertible non-atomic ergodic measure-preserving system $(X, \mu, T)$ with  measure-theoretic entropy strictly less than
the topological entropy of $S$ there exists an embedding of $(\X,\mu,T)$ into $(Y,S)$,
and we say that $(Y,S)$ is {\df fully universal} if the embedding
can be chosen so that the push-forward of the measure on $\X$ is fully supported on $Y$.
The Krieger finite generator
theorem \cite{Krieger1,Krieger1e} says that the full shift on a finite number of symbols is universal.
We prove the following extension of Krieger's theorem.

\begin{theorem}
\label{result-toral-gen}
Toral automorphisms are universal.
\end{theorem}

Following Lind \cite{lind}, we say that a toral automorphism is   {\df quasi-hyperbolic} if its
associated matrix $A$ has no roots of unity as eigenvalues, and {\df hyperbolic} if $A$ does not
have an eigenvalue of modulus $1$.  Lind and Thouvenot \cite[Section 5]{LinTho} proved that
hyperbolic (two-dimensional) toral automorphisms are fully universal  and asked whether
the same is true in the quasi-hyperbolic  case, and in the more general case
of an automorphism of a compact group (that is ergodic with respect to Haar measure).
Lind and Thouvenot made use of the fact that hyperbolic
toral automorphisms can be represented as irreducible shifts of finite type; this is not true in
the non-hyperbolic case \cite[Section 6]{lind-spec}, \cite[Theorem 4]{lind}.    We prove an
affirmative answer to their question, which easily implies Theorem \ref{result-toral-gen}.

\begin{theorem}
\label{result-toral}
A quasi-hyperbolic toral automorphism is fully  universal.
\end{theorem}

Lindenstrauss and Schmidt \cite{MR2121533,MR2191216}  studied the structure of invariant measures for quasi-hyperbolic toral automorphisms  and established strong structural properties for any such measure.   For this reason, it may come as a surprise that these systems have the universality property that we establish in this paper.

Theorem \ref{result-toral} will be proved as part of a more general result
which we state in Section \ref{non-exp}.
The result has three conditions, one of which is a form of specification.

Let $S$ be a self-homeomorphism of a compact metric space $(Y,\metric)$.     Suppose that for every $\e >0$ there
exists a function   $L_{\e} : \Z^{+} \to \Z^{+}$   such that given a finite number
of points $y_1, \ldots, y_n \in Y$
and finite sequence of integers $a_1 \leq b_1<a_2 \leq b_2
\cdots a_n \leq b_n$ with $a_i - b_{i-1} \geq L_{\e}(b_i-a_i)$
for all $i \in[2,n]$, there is a $y \in Y$ with $\metric(S^k y, S^k y_i)
\leq \e$ for all $k\in[a_i,b_i]$ and $i \in [1,n]$.
We call $L_{\e}$ a {\df gap} function.  If the gap function $L_{\e}$ satisfies $\fracc{L_\e(m)}{m} \to 0$ as $m \to \infty$,
then we say that $S$ satisfies
{\df almost weak specification};  if the function $L_{\e}$ is a constant function, then $S$ satisfies
{\df weak specification}, and in addition, if $y$ can be chosen to be a periodic point, then
$S$ satisfies {\df specification} (for background see \cite{Bowen2, thom, Dateyama,  Sigmund, MR2529890}).  Without loss of generality, we will always assume that a gap function is non-decreasing.

Marcus \cite{Marcus} proved that  quasi-hyperbolic
toral automorphisms satisfy almost weak specification.
Let us remark that  a toral automorphism is ergodic with respect to Haar
measure if and only if it is quasi-hyperbolic  \cite{Halmos-autos}.

\begin{conjecture}
\label{con}
A self-homeomorphism with almost weak specification on a compact metric space is (fully) universal.
\end{conjecture}

Recently, we also proved that the time-one map of a geodesic flow on a compact
surface of negative curvature
is universal \cite{QuaSoo}.   With the help of the symbolic dynamics for geodesic flows developed by Bowen
\cite{Bowen3} and Ratner \cite{Ratner}, it was sufficient to show that the time-one map of a
topologically weak-mixing suspension flow over an
irreducible subshift of finite type
is universal;
in our proof of this result we were aided by the symbolic nature of the suspension flow and the
fact that the time-one map satisfies (weak) specification \cite[Proposition 5]{QuaSoo}.

Krieger also proved that mixing subshifts of finite type are universal \cite{kri-icm}
(see also the proof given by Denker \cite[Theorem 28.1]{MR0457675}, and the  `Borel' embedding given by Hochman \cite{MR3077948},\cite{MR3125641}, \cite{buzzi}).
We  also prove the following generalization of Krieger's theorem,
and Lind and Thouvenot's result \cite[Theorem 2]{LinTho} that every mixing
subshift of finite type is fully universal.

\begin{theorem}
\label{result-sym}
A subshift with almost weak specification on a finite number of symbols  is fully universal.
\end{theorem}

We will prove a weaker version of Conjecture \ref{con}, Theorem \ref{result-gen},
from which Theorems \ref{result-toral} and \ref{result-sym} will follow.
We will require some additional conditions that are  satisfied
in Theorems \ref{result-toral} and \ref{result-sym}.   We will give the precise statement of
these additional conditions in the next section.

In our proof of Theorem \ref{result-gen}, we make use of an idea of
Burton and Rothstein \cite{BurRot}, further developed in work
of Burton, Keane and Serafin \cite{BurKeaSer}, and Downarowicz and Serafin \cite{easyorn}
that involves producing the required injections using the
Baire Category theorem.

In Section \ref{section-examples} we will prove Theorems \ref{result-toral} and  \ref{result-sym} using our more general
result.
We will also generalize Theorem \ref{result-toral} to  include any automorphism of a compact
metric abelian group that is ergodic with respect to Haar measure;   with this generalization we will examine a generalization of an  isomorphism invariant proposed by Halmos (see Corollary \ref{weiss}).   In Section \ref{section-soft}
we set the stage for the proof of our more general result; the proof will be carried out  in the remaining sections.

\section*{Acknowledgments} We would like to thank Jean-Paul Thouvenot for
introducing us to the problem and Benjy Weiss for his interest in our paper and for pointing out Corollary \ref{weiss}.  We are also grateful for the  referee's useful suggestions.

\section{Non-expansive homeomorphisms and the small boundary condition}
\label{non-exp}

We recall in this section some basic tools to deal with non-expansive homeomorphisms.
Let $(X, \metric)$ denote a compact metric space.    For $ r>0$, we let $B(x,r)$ denote the open
 ball about a point $x \in X$, and for a subset $A \subset X$, we let  $B(A,r) := \bigcup_{x \in A} B(x,r)$.
If $\mathcal P$ is a finite (Borel-)measurable partition of $X$
its {\df diameter} is defined to be the maximum of the diameters of the elements of the partition.
The set $\partial\mathcal P$ is the union of the topological boundaries of the elements of $\mathcal P$
and $\partial_r\mathcal P$ denotes $\overline{B(\partial \mathcal P,r)}$.
We recall that a partition is said to be {\df generating} if for any distinct pair of
points, $x$ and $y$, there exists an $n\in\Z$ such
that $T^nx$ and $T^ny$ lie in different elements of $\mathcal P$.

A self-homeomorphism $T$ of a metric space $(X, \metric)$ is  {\df expansive} if there exists
a $\delta>0$ (the {\df expansiveness constant})
such that  for all $x,y \in X$ if $\metric(T^nx,T^ny)<\delta$ for all $n\in\Z$, then $x=y$.
Expansive homeomorphisms have many generating partitions:
Indeed any partition of diameter less than the expansiveness constant is generating.
We recall that any subshift (that is the restriction of the shift map
to a non-empty closed shift-invariant subset of the full shift) is expansive.

We write $M_T(X)$ for the collection of $T$-invariant Borel probability measures on $X$.
Another desirable feature of expansive homeomorphisms is that the {\df entropy
functional} given by $\mu\mapsto h_\mu(T)$ sending an invariant measure
to its measure-theoretic entropy is upper semi-continuous with respect to the weak$^*$ topology
(see Lemma \ref{weak-metric} for an explicit metric which generates the weak$^*$ topology).
In
order to consider non-symbolic spaces we will make use of two successive weakenings
due to Bowen \cite{Bowen}
and Misiurewicz \cite{Misiurewicz2} which allow us to recover the upper semi-continuity of the entropy functional.

For each $\delta >0$, define the set
\begin{equation*}
\Gamma_\delta(x):= \{y\in X\colon \metric(T^jx,T^jy)<
\delta\text{ for all $j\geq 0$}\}.
\end{equation*}
If there is a constant $\delta>0$ such that $h_\mathrm{top}(T,\Gamma_\delta(x))=0$ for all $x\in X$,
then $T$ is {\df entropy expansive}, and
if $\sup_{x\in X}h_\mathrm{top}(T,\Gamma_\delta(x))\to 0$ as $\delta\to 0$, then $T$ is
{\df asymptotically entropy expansive}.

That these are strict weakenings of expansiveness is seen by simple examples such as the
identity map, twist maps such
as $T(x,y)=(x,y+x)\bmod 1$.
More substantial examples are given by the following lemma.

\begin{lemma}[Bowen \cite{Bowen} Example 1.2]
\label{toral-bowen}
Toral automorphisms are entropy expansive.
\end{lemma}

Thus while hyperbolic toral automorphisms are expansive, in general toral automorphisms are only entropy expansive.
   \begin{comment}
  In Lind's paper \cite[Theorem 1]{Lind}, he gives
    an explicit constant $\delta$ so that $T$ acts as an isometry on $\Gamma_\delta(x)$
    for each $x$ (with respect
    to a suitable metric based on a norm on $\mathbb{R}^d$ that he also constructs). It follows that
    the topological entropy of $T$ restricted to $\Gamma_\delta(x)$ is 0 as required.
\end{comment}

\begin{lemma}[Misiurewicz \cite{Misiurewicz2} Corollary 4.1]
\label{lem:Misiurewicz}
    Let $T$ be an  asymptotically entropy expansive self-homeomorphism of a
    compact metric space $X$.
    Then the entropy functional  is upper
    semi-continuous with respect to the weak$^*$ topology on $M_T(X)$.
\end{lemma}

Bowen proved Lemma \ref{lem:Misiurewicz} in the case where $T$ is entropy expansive.
For an example of a map that satisfies specification, but is not asymptotically
entropy expansive and does not have an upper semi-continuous entropy functional
see  \cite[Page 952]{MR2322186}.

Another property that we need is sometimes called the small boundary property or the
existence of an essential partition. A metric space is {\df $\boldsymbol{d}$-dimensional}
(i.e.\  it has topological dimension $d$)
if there exists an open cover with arbitrarily fine diameter with the property that the intersection of
any $d+2$ distinct elements is empty. A self-homeomorphism, $T$, of a compact metric
space $X$ is said to have the
{\df small boundary property} if for each $\delta>0$, there exists a partition
$\mathcal P$ of $X$ such that $\mathrm{diam}(\mathcal P)<\delta$ and $\mu(\partial \mathcal P)=0$
for each $T$-invariant measure  $\mu$.
Notice that if $\mu(\partial \mathcal P)=0$, then $\mu(\partial(T^{-1}\mathcal P))=0$. Also if
$\mu(\partial \mathcal P)=0$ and $\mu(\partial\mathcal Q)=0$, then $\mu(\partial(\mathcal P
\vee \mathcal Q))=0$.

Both the  small boundary property and asymptotic
entropy expansiveness are important properties in the theory of symbolic extensions and
entropy structure; for more information see the recent book of Downarowicz \cite{MR2809170}
and the recent articles of  Boyle and Downarowicz  \cite{MR2047659},
Burguet \cite{Burguet},  Downarowicz
\cite{MR2177182, Downarowicz}, and  Lindenstrauss \cite{Lindenstrauss}.    We make use
of the following result of Kulesza.

\begin{lemma}[Kulesza \cite{Kulesza} Lemma 3.7]
\label{Kulesza}
A self-homeomorphism of a compact finite-dimensional metric space
with the property that the periodic points form a zero-dimensional set has the small boundary
property.
\end{lemma}

Krieger \cite{kri-icm} raised the question of when a homeomorphism of a compact metric space is
universal.  We have the following partial answer.

\begin{theorem}
\label{result-gen}
A  self-homeomorphism of a compact metric space is fully universal whenever it satisfies:
\begin{enumerate}
\item
\label{aws-spec}
almost weak specification;
\item
\label{cond-a-h}
asymptotic entropy expansiveness; and
\item
\label{small-b-c}
the small boundary property.
\end{enumerate}
\end{theorem}

Notice that Theorem \ref{result-sym} follows immediately from Theorem \ref{result-gen}.
%We will make use of the technical  conditions \eqref{cond-a-h} and \eqref{small-b-c} in the proofs of Lemma \ref{baire-lemma} and  Proposition \ref{enl-open} \eqref{open}, respectively.
  Let us also remark that in Theorem \ref{result-gen} we may  replace condition \eqref{cond-a-h} with the weaker condition that the entropy functional is upper semi-continuous.

Benjy Weiss \cite{Weiss-person2} has informed us that he has made progress on Conjecture \ref{con}; in particular, he has generalized Theorem \ref{result-gen} to hold with only conditions \eqref{aws-spec} and \eqref{small-b-c}, and also has a version for $\Z^d$-actions.  For background on universality for $\Z^d$-actions see \cite{MR1844076}.

\section{Examples}
\label{section-examples}

\subsection{Subshifts}
\label{subshifts}

We call a subshift
{\df non-trivial} if it does not consist of a finite set of points.
Subshifts have a natural generating
partition given by $\mathcal P:= \ns{P_1, \ldots P_n}$, where $P_i = \ns{x \in X: x_0 =i}$ for all $1 \leq i \leq n$.

Familiar examples of subshifts that satisfy specification are  mixing
subshifts of finite type, and mixing sofic subshifts.
However, there are many subshifts that do not satisfy almost weak specification,
but are still universal.

\begin{corollary}
\label{cor-sym}
If a subshift contains subshifts that satisfy almost weak specification and have topological entropy
arbitrarily close to the original topological entropy, then it is universal.
\end{corollary}

\begin{proof}
Immediate from Theorem \ref{result-sym} and the definition of universality.
\end{proof}

\subsection{Toral automorphisms}

Lind \cite[Section 3]{lind-spec} classifies hyperbolic and
quasi-hyperbolic toral automorphisms into
three disjoint classes according to the spectral properties of the associated matrix.
These classes  correspond precisely to  the various forms of specification that were
defined in Section \ref{intro}.    Recall that
 the associated matrix for a hyperbolic
toral automorphism has no roots of unity, and hence all quasi-hyperbolic toral automorphisms
in dimensions two or three are hyperbolic.
Bowen \cite{Bowen2}  proved that hyperbolic
toral automorphisms satisfy specification. The associated matrix for a {\df central spin}
automorphism has some eigenvalues on the unit circle, and the Jordan blocks for these
eigenvalues have no off-diagonal $1$'s,  and all other quasi-hyperbolic toral
automorphisms are {\df central skew} automorphisms.    Lind \cite[Theorem (ii)]{lind-spec}
proved that central spin automorphisms (which only occur in dimensions $4$ or higher)
satisfy weak specification, but never satisfy specification.
Lind \cite[Theorem (iii)]{lind-spec} also proved that central skew automorphisms
(which only occur in dimensions $8$ or higher) never satisfy weak specification,
but nevertheless Marcus \cite{Marcus} proved that they still must satisfy almost
weak specification.     It is easy to give explicit examples of all the different
types of quasi-hyperbolic toral automorphisms (see for example \cite[ Section 3]{lind}).

\begin{proof}[Proof of Theorem \ref{result-toral}]
By Lemma \ref{lem:Misex}, we know that toral automorphisms are  entropy expansive.
Since we assume quasi-hyperbolicity, almost weak specification is satisfied \cite{Marcus},
and furthermore an easy linear algebra argument shows the periodic points are exactly
the points on the torus with rational coordinates; it is an elementary fact
\cite[Proposition 1.2.4]{engelking} that countable sets have zero-dimension.
Thus Theorem \ref{result-toral} follows from Lemmas \ref{toral-bowen} and \ref{Kulesza} and  Theorem \ref{result-gen}.
\end{proof}

We will use the following elementary factorization lemma to prove Theorem \ref{result-toral-gen}.

\begin{lemma}
\label{lin-alg}
Any toral automorphism is the direct product of a quasi-hyperbolic toral automorphism
and a toral automorphism of zero topological entropy.
\end{lemma}

\begin{proof}
The lemma is a consequence of the following facts.    The topological  entropy of a toral
automorphism is given by the sum of the logarithms of the moduli of the associated
eigenvalues of modulus at least one (see for example \cite[Corollary 16]{bowen1971}).
A toral automorphism with a characteristic polynomial $f$  that can be expressed as a
product $f=gh$ such that $g$ and $h$ are polynomials over $\Z$ with $\gcd(g,h) =1$
is the direct product of toral automorphisms with characteristic polynomials $g$ and $h$.
A polynomial over $\Z$ with a root of unity is the product of a polynomial over $\Z$ with
only roots of unity and a polynomial over $\Z$ without roots of unity.
\end{proof}

\begin{proof}[Proof of Theorem \ref{result-toral-gen}]
Let $S$ be an automorphism of the $d$-dimensional torus $\mathbb{T}^d$.   By Lemma
\ref{lin-alg}, let $S=  S^q \times S^o$, where $S^q$ is a quasi-hyperbolic toral automorphism
of a $d_1$-dimensional torus $\mathbb{T}^{d_1}$  with the same topological entropy as $S$,
and $S^o$ is a zero entropy toral automorphism of a $d_2$-dimensional torus $\mathbb{T}^{d_2}$.
Note that if $T$ is an invertible  ergodic-measure preserving transformation of a probability space
$(\X, \mu)$ and $\Psi$ is an embedding of $(\X, \mu, T)$ into $(\mathbb{T}^d,S^q)$, then the map
$\Psi'(y,z) :=  (\Psi(y), 0)$ gives an embedding of $(X, \mu,T)$ into $(\mathbb{T}^{d_1} \times
\mathbb{T}^{d_2}, S^q \times S^o)$, since $0$ is a fixed point of any toral automorphism.
Thus the universality of $(\mathbb{T}^d, S)$ follows from the universality of $(\mathbb{T}^{d_1},S^q)$
and Theorem \ref{result-toral}.
\end{proof}

We give the following easy application of Theorem \ref{result-toral-gen}.
Let   $T$ be a $C^{1+ \delta}$ ($\delta >0$) diffeomorphism of a compact smooth manifold.   Let
\[
E(T):=\ns{h_{\mu}(T) : \mu\text{ is an ergodic invariant measure for $T$}}.
\]
Katok conjectures  $E(T) \supset [0, h_\mathrm{top}(T))$ and proved this to be true  in the case where $T$ is
a diffeomorphism of a two-dimensional surface \cite{Sun,MR804774}; the proof follows immediately
from   \cite[Theorem 4]{MR804774}.   Recently, Sun \cite{Sun}
showed that $E(T)$ is dense in $[0, h_{\mathrm{top}}(T)]$ for a toral automorphism $T$;
Theorem \ref{result-toral-gen} extends the result of Sun and  implies that Katok's conjecture is
true for toral automorphisms.

\subsection{Automorphisms of a compact metric abelian group}

Toral automorphisms form an important class of examples of automorphisms of compact metric
abelian groups.  Note that an
automorphism always preserves the Haar measure.   Generalizing the earlier result of
Marcus \cite{Marcus}, Dateyama \cite{Dateyama} proved that an automorphism of a compact
metric abelian group satisfies almost weak specification if and only
if it is ergodic with respect to the Haar measure.

Ergodicity in the setting of an automorphism $S$ of a compact metric abelian group $Y$
still has an algebraic characterization.   Recall that the character group of $Y$
is the set of all continuous group homomorphisms of $Y$ into the complex unit circle
with multiplication and denoted by $\hat{Y}$.   Let $U_S:L^2 \to L^2$ be the Koopman
representation given by $U_S(\chi) = \chi \circ S$, where $L^2$ is the set of all
square integrable functions on $Y$ with respect to Haar measure.   Since $S$ is
also a group homomorphism, we have that $U_S : \hat{Y} \to \hat{Y}$.   Halmos
\cite{Halmos-autos} proved that $S$ is ergodic with respect to Haar measure if and only if
$U_S$ has no finite orbits other than the trivial character; that is for all
$\chi \in \hat{Y}$ and all $n \in \Z^{+}$, we have  $U_S^n(\chi) =\chi$ if and
only if $\chi=1$.

\begin{theorem}
\label{result-group}
An automorphism of a compact finite-dimensional metric abelian group with finite topological entropy
whose Koopman representation  has no finite orbits on the character group, other than the trivial
character, satisfies almost weak specification and is universal.
\end{theorem}

The proof of Theorem \ref{result-group} is similar to the proof of Theorem \ref{result-toral}
making use of slightly more technical tools to verify the conditions of Theorem \ref{result-gen}.
The following lemma substitutes entropy expansiveness with
asymptotic entropy expansiveness.

\begin{lemma}[Misiurewicz  \cite{Misiurewicz2}  Example 7.1]
\label{lem:Misex}
An endomorphism of a compact group with finite topological entropy 
is asymptotically entropy expansive.
\end{lemma}

Finally, we also need to argue that the periodic points are countable.   We do not have an
elementary argument for this, and instead will refer to a recent paper of Miles
\cite{miles-per} who gives a formula for the number of periodic points of each finite order
of a ergodic finite entropy automorphism of a finite-dimensional compact abelian group; we will
not need to use the full force of this formula.

\begin{proof}[Proof of Theorem \ref{result-group}]
By Lemma \ref{lem:Misex}, we have asymptotic entropy expansiveness.   Since the Koopman
representation has no non-trivial finite orbits, we have that the automorphism is
ergodic with respect to Haar measure \cite{Halmos-autos} and thus satisfies almost weak
specification \cite{Dateyama}.     Since the automorphism is ergodic and has finite entropy,
by \cite[Lemma 4.3]{miles-per}, the number periodic points of each finite order is finite
thus the number of periodic points is countable and has zero-dimension
\cite[Proposition 1.2.4]{engelking}.   Theorem \ref{result-group}  now follows from
Theorem \ref{result-gen} and Lemma \ref{Kulesza}.
\end{proof}

\subsection{Halmos' invariant}

In his book \cite{Halmos-ergodic}, Halmos proposed an invariant of two-dimensional
toral automorphisms.   Let $T: X \to X$ be an invertible ergodic measure preserving transformation.   If $p(t)=\sum_k a_k t^k\in\Z[t]$ and $f:X \to \mathbb{T}$,  then let $p(T)[f]:X \to \mathbb{T}$ be given by   $\sum_k a_kf\circ T^k$.
Halmos noted that if $T:\mathbb{T}^2 \to \mathbb{T}^2$ is a toral automorphism, $f\colon \T^2\to\T$ is one of the coordinate functions, and  $p$ is the characteristic polynomial of the automorphism,
then $p(T)[f]=0$,
even though the function $f$ is a non-constant function mapping into $\mathbb T$.
\begin{comment}
Here, by $p(T)$, if $p(t)=\sum a_k t^k\in\Z[t]$, we mean $p(T)f=\sum_k a_kf\circ T^k$.
\end{comment}
Clearly the existence of a non-constant $\T$-valued $f$ such that $p(T)[f]=0$
almost everywhere is an isomorphism invariant. Halmos asked whether this invariant (with the
polynomials $p$ taken to be of the form $p(t)=t^2-at+1$) distinguishes
two-dimensional toral automorphisms with distinct characteristic polynomials.

A natural generalization
is to set $\poly(X,T,\mu):=\{p\in\Z[t]\colon$ $p(t)$ is irreducible over $\Z$ and
there exists a non-constant $\mathbb T$-valued $f$ such that $p(T)[f]=0$, $\mu$-a.e.$\}$.
Notice that for a reducible polynomial $q(t)\in\Z[t]$, there
is a non-constant $\T$-valued $f$ such that $q(T)[f]=0$ $\mu$-a.e. if and only if one of its irreducible
factors belongs to $\poly(X,T,\mu)$.

The following corollary establishes that $\poly(X,T,\mu)$ is no stronger  an invariant
than the rational spectrum (in the sense of its ability to
distinguish transformations).  We would like to thank Benjy Weiss \cite{Weiss-person1} for pointing
this corollary out to us, and for suggesting we include it in this paper.

%\newpage

\begin{corollary}[Weiss]\label{weiss}
Let $(X,T,\mu)$ and $(X',T',\mu')$ be ergodic invertible measure-pres\-erving transformations of
non-atomic  probability spaces. If $T$ and $T'$ have the same rational spectrum,
then $\poly(X,T,\mu)=\poly(X',T',\mu')$.
\end{corollary}

\begin{proof}

Let $(X,T,\mu)$ be an ergodic invertible measure-pres\-erving transformation of
non-atomic  probability space.  Let $\mathrm{NCI}$ denote the collection of non-cyclotomic irreducible polynomials in $\Z[t]$.   For a positive integer $m$, let $S=S_m$ denote the transformation $S(t)=t+1\bmod m$ of $[m]=\{0,\ldots,m-1\}$
and $c=c_m$ denote normalized counting measure on $[m]$.  We claim the following, from which the statement of the
corollary holds.
\begin{align}
\label{NCI}
\poly(X,T,\mu)&=\big(\mathrm{NCI}\setminus\{\pm1,\pm t\}\big) \
\cup\bigcup_{m\in R(T)}\poly([m],S,c).
\end{align}

  We first show that  $\mathrm{NCI} \setminus\{\pm1,\pm t\} \subseteq   \poly(X,T,\mu)$.   Let $\mathbb{T}^\Z$ be equipped with the shift map $\sigma$.  For $p(t) \in \mathbb{Z}[t]$, let $p(\sigma): \mathbb{T}^{\Z} \to \mathbb{T}^{\Z}$ be given by   $\sum_k a_k \sigma^k$, and $I_p:=\{z\in\mathbb T^\Z\colon
p(\sigma)z=0\}$.   Note that if $ p(t) \in \mathrm{NCI} \setminus\{\pm1,\pm t\}$, then $I_p$ has non-zero topological entropy and  satisfies the conditions of Theorem \ref{result-group}.   Thus provided $(X,T,\mu)$
has entropy smaller than $h_\text{top}(I_p,\sigma)$, there is a non-constant equivariant
map $\pi\colon (X,\mu)\to (I_p,\sigma)$.
If $h_\mu(T)\ge h_\text{top}(I_p,\sigma)$, then we first take a two-element partition $\mathcal P$ of $X$
such that $H(\mathcal P)<h_\text{top}(I_p,\sigma)$ and consider the factor of $X$ induced by $\mathcal P$.
This has strictly smaller entropy than $h_\text{top}(I_p,\sigma)$, which is then embedded as before.
In either case, we end up with a non-constant equivariant map $\pi$ from $(X,T,\mu)$ to $(I_p,\sigma)$.
Taking $f(z)=z_0$, where $f: \mathbb{T}^{\Z} \to \mathbb{T}$, we see $p(T)[f\circ \pi]=0$ for $\mu$-a.e.\ $x\in X$,
so that $p\in \poly(X,T,\mu)$.

Next, it suffices to argue that both sides of \eqref{NCI}
contain the same cyclotomic polynomials.
That $\poly([m],S,c)$ is contained in $\poly(X,T,\mu)$ for $m\in R(T)$ is clear by lifting a
function defined on $[m]$ to $X$ through the factor map.      We denote the $n$th cyclotomic polynomial by $\Phi_n$.
For the converse, suppose that $\Phi_n\in\poly(X,T,\mu)$. Let $f\colon X\to\T$ be non-constant,
but such that $\Phi_n(T)[f]=0$.
Since $\Phi_n(T)[f]=0$ and $\Phi_n(t) \ | \ (t^n-1)$, we have that
 $f\circ T^n=f$ $\mu$-a.e.   Let $m$ be the period of the sequence $(f\circ T^j)$, so that
$m\in R(T)$. Then $(X,T,\mu)$ factors
onto $[m]$ and transferring the function $f$ to $[m]$ shows that $\Phi_n\in \poly([m],S,c)$ as required.
\end{proof}

We remark that we do not know which cyclotomic polynomials are contained in $\poly([m],S,c)$.    If $\ell>1$ is an integer factor of $m$, then one can check that
$\Phi_\ell\in \poly([m],S,c)$ (as witnessed by the function $g(x)=\cos(2\pi (x-\alpha)/\ell)$
for $\alpha$ irrational --
notice that the $\alpha$ guarantees that $g$ is not constant for $\ell>1$).
However, $\poly([m],S,c)$ may contain $\Phi_n$'s for integers $n$ that
are not factors of $m$ also.  For example the function $g(0)=\frac13$, $g(1)=\frac23$ shows that
$\Phi_6\in \poly([2],S,c)$.

\subsection{Geodesic flows and suspension flows}

The small boundary property is a technical condition that is convenient for our method
of proof, but is  not necessary for universality.    The time-one map of a geodesic flow on a compact
surface of negative curvature is entropy expansive  \cite[Example 1.6 $^{*}$]{Bowen}
and time-map of a suspension flow can  also be shown to be entropy expansive by
\cite[Example 1.6]{Bowen}.   The time-one map of a geodesic flow of a
compact surface of negative curvature
satisfies specification \cite[(Example) F]{Sigmund}   and the time-one map of a topologically
weak-mixing suspension flow satisfies (weak)  specification \cite[Proposition 5]{QuaSoo}.
However, clearly the time-one   maps of these flows may have uncountably many periodic
{\em points}, and moreover,   it is not difficult to construct examples of topologically
weak-mixing suspensions   flows that do not have the small boundary property.
In spite of this, we proved universality for
these flows \cite[Theorems 1 and 2]{QuaSoo}.

\subsection{Non-examples}

We give an example to show that it is too much to ask that the strict entropy difference
in the definition of universality can be relaxed.    The same example also shows that
universality can not always be satisfied; this fact also easily follows from
the Krieger-Jewett theorem \cite{krieger2}.

Let $S$ be the full-shift on $\Omega:=\ns{0,1}^{\Z}$ endowed with the Bernoulli measure
$\zeta$ that is the unique measure of maximal  entropy--$\log(2)$.
Let $S'$ be an irrational rotation of the circle $\mathbb{T}^1$.
The invertible measure-preserving transformation $S \times S'$ is ergodic and has entropy $\log(2)$.
The invertible measure-preserving system  $(\Omega \times \mathbb{T}^1, S \times S')$ can not be
embedded within $(\Omega, S)$, since the embedding would
yield a measure-theoretic isomorphism of $S \times S'$ and $S$; this  implies that $S$
has a non-trivial zero entropy factor which contradicts that fact that $S$ is a
$K$-automorphism  \cite[Part 1, Paper 2]{MR2766434}.   Similarly, if the topological dynamical
system $(\Omega \times \mathbb{T}^1,S \times S')$
were universal, then every measure-preserving automorphism with entropy strictly
less than $\log(2)$ would have a non-trivial zero entropy factor.

\section{Soft methods}
\label{section-soft}

\subsection{Basic idea}

Roughly, the idea of  Burton and Rothstein \cite{BurRot},
is that one defines an notion of $\epsilon$-approximate
embedding  (see Proposition \ref{enl-open})  with the property that a point lying in the intersection
of the $\epsilon$-approximate embeddings is
a true embedding. One then introduces a Polish space (a separable completely metrizable topological space)
consisting of potential injections (these are in fact
joinings, see Lemma \ref{baire-lemma}) and shows that the $\epsilon$-approximate embeddings form a
dense open subset of the set of potential embeddings.
Baire's theorem  gives the desired result.  The essential (and surprising) feature of this idea is that in a topological
sense, almost any candidate works. Burton, Keane and Serafin reproved the Krieger generator
theorem \cite{Krieger1},  the Sinai factor theorem \cite{Sinai, MR2766434}, and the
Ornstein isomorphism theorem \cite{Ornstein}.

Keane and Smorodinsky \cite{keanea, keaneb} gave an explicit proof of the Ornstein
isomorphism theorem.   Similarly, there is an alternative approach to producing embeddings that
we used in our earlier paper \cite{QuaSoo}.
In that approach, one carefully
produces a sequence of approximate embeddings, takes a pointwise limit, and proves that this limit
has the desired properties.    (However, our construction does not yield a finitary map,
unlike those of Keane and Smorodinsky.)

\subsection{Joinings and Baire category}

In this subsection, we make precise the ideas outlined above.
We will make use of the following
explicit metric giving rise to the weak$^*$-topology on measures.
Let $Z$ be a compact metric space.   We let $\mathrm{Lip}_1(Z)$ denote the
space of all real-valued Lipschitz continuous functions on $Z$
taking values in $[0,1]$ with Lipschitz constant
no greater than $1$.     Let $\mathcal M(Z)$ denote the collection of all Borel
probability measures on $Z$.   Write $\mu(f):= \int f d\mu$ for
$\mu \in \mathcal M(Z)$ and an integrable function $f: Z \to \R$.

\begin{lemma}[Distance giving weak$^*$-topology on measures]
\label{weak-metric}
    Let $Z$ be a compact metric space.  Define a metric $\metricstar$ on $\mathcal M(Z)$ as follows:
    \begin{equation*}
        \metricstar(\mu,\nu):=\sup_{f\in \mathrm{Lip}_1}\left|\mu(f) - \nu(f) \right|.
    \end{equation*}
    The metric $\metricstar$ gives rise to the weak$^*$ topology on $\mathcal M(Z)$.
\end{lemma}

The proof is a simple adaptation of the proof in Dudley's book (\cite{Dudley}, Theorem 11.3.3) to take account of the
fact that functions in $\mathrm{Lip}_1$ are required to take values in $[0,1]$.

In general, a measure-preserving transformation is defined on a measure space. In order
to use the Baire category machinery, we will need the space to be embedded to be a metric space.
The following lemma allows us to assume that the transformation to be embedded lives on a metric space.

\begin{lemma}
\label{kriegerized}
Let $S$  be a self-homeomorphism of a compact metric space $Y$.  Then $(Y,S)$ is
(fully) universal if and only if  there exists an
embedding of every non-trivial ergodic subshift $(X,T,\mu)$ with $h_{\mu}(T) < h_{\mathrm{top}}(S)$ into
$(Y,S)$  (such that the push-forward of $\mu$ is fully supported on $Y$).
\end{lemma}

\begin{proof}
The only if part is clear from the definition of universality. For the converse
by the Krieger generator theorem, for an arbitrary  invertible measure-preserving transformation $T_0$
of the probability space $(X_0,\mu_0)$, there is a measure-theoretic isomorphism between $(X_0,T_0,\mu_0)$
and a subshift $(X,T,\mu)$ (which, of course, preserves the entropy). By the assumption of the lemma, this subshift
may be embedded into $(Y, S)$.  Composing the isomorphism and the embedding gives the result.
\end{proof}

Thus by Lemma \ref{kriegerized}, we may always assume that the space to be embedded is a
subshift on a finite number of symbols.

Let $(X, \mathcal F, \mu, T)$ be a non-trivial ergodic subshift with invariant measure $\mu$ and Borel
$\sigma$-algebra $\mathcal F$.    Let $S$
be a self-homeomorphism of a compact metric space $Y$ with Borel $\sigma$-algebra $\mathcal B$.
By a {\df $\boldsymbol{\mu}$-joining}, we mean a $(T\times S)$-invariant
measure on the product metric space  $X \times Y$ (with the product $\sigma$-algebra
$\mathcal F \otimes \mathcal B)$) whose
$X$-marginal is $\mu$.      (Note that
unlike the standard definition of a joining, we do not make any requirement
on the $Y$-marginal;   for more background on joinings see \cite{ Glasner,Rud, Rue}.)
We let $J_\mu(T,S)$ denote the space of $\mu$-joinings.   It is well known that $J_\mu(T,S)$ is a (non-empty)
compact metric (hence complete) space with the weak$^*$ topology.
We let $\pi_1$ and $\pi_2$
be the coordinate projections from $X\times Y$ to $X$ and $Y$, respectively,  and denote the push-forward maps
by $\pi_1^*$ and $\pi_2^*$, so that for $\xi\in J_\mu(T,S)$, we have $\pi_1^*(\xi)=\mu$.
Before we prove topological properties regarding subsets of the space of joinings, we first state an
elementary but useful fact which motivates how  joinings are related to embeddings.

Let $(Z, \mathcal G, \xi, U)$ be a
measure-preserving system.   Let $\mathcal F, \mathcal B \subset \mathcal G$ be  sub-$\sigma$-algebras.
We write $\mathcal F {\subset}\mathcal B \mod \xi$ if for each $F \in \mathcal F$, there is a $ B \in \mathcal B$
such that the $\xi$-measure of the symmetric difference is zero.   Similarly, for each
$ \e >0$, we write  $\mathcal F\stackrel{\epsilon}{\subset}\mathcal B \mod \xi$,
if for each $F \in \mathcal F$, there is a $ B \in \mathcal B$
such that the $\xi$-measure of the symmetric difference is strictly less than $\epsilon$.
For a partition $\mathcal P$, we denote by $\sigma(\mathcal P)$ the finite $\sigma$-algebra
that it generates.
Note that $\mathcal P$ is a generating partition for $Z$ (mod $\xi$)  if
and only if  $\mathcal G \subset \bigvee_{i\in\Z}T^{-i}\mathcal P \mod \xi$.
Also let  $\mathcal T_Z=\{\emptyset,Z\}$
denote the trivial $\sigma$-algebra on $Z$.

\begin{proposition}
\label{rue}
Let $(X, \mathcal F, \mu, T)$ be a non-trivial ergodic subshift with invariant measure $\mu$ and
Borel $\sigma$-algebra $\mathcal F$.   Let $S$ be a
self-homeo\-morphism of a compact metric space $Y$ with Borel $\sigma$-algebra $\mathcal B$.   Let
$\xi$ be a $\mu$-joining satisfying the following conditions
\begin{enumerate}[(i)]
\item
$\mathcal T_X\otimes \mathcal B \subset \mathcal F \otimes \mathcal T_Y  \mod \xi$;
\item
$\mathcal F \otimes \mathcal T_Y \subset   T_X\otimes \mathcal B \mod \xi.$
\end{enumerate}
Then there exists an embedding $\Psi: X \to Y$ of $(X,T,\mu)$ into $(Y,S)$ such that
\begin{equation}
\label{define}
\xi(F \times B) = \mu(F \cap \Psi^{-1}(B)) \text{ for all }
(F, B) \in \mathcal F \times \mathcal B.
\end{equation}
Similarly, if $\Psi$ is an embedding of $(X, T, \mu)$ into $(Y,S)$, then $\xi$ defined by \eqref{define}
satisfies the first two conditions.  \end{proposition}

\begin{proof}
Proposition \ref{rue} follows easily from \cite[Theorem  2.8]{Rue}.
\end{proof}
Thus if $\xi$ satisfies the two containment conditions in Proposition \ref{rue},  we also say
that $\xi$ is an {\df embedding}. The Burton--Rothstein argument, that we follow, defines
approximate embeddings and works by showing that there are a large collection of approximate
embeddings at each scale.

\begin{lemma}[The Baire Space]
\label{baire-lemma}
Let $(X, \mu, T)$ be  a non-trivial ergodic subshift with invariant measure $\mu$, and
let $S$ be a self-homeomorphism
of a compact metric space $Y$ that is  asymptotically entropy expansive.
Let $h_{\mu}(T) < h_{\mathrm{top}}(S)$.
The space defined  by
\begin{equation*}
\label{baire}
\M_0:=\{\xi\in J_\mu(T,S)\colon \xi \text{ is ergodic and }h_{\pi_2^*(\xi)}(S) \ge h_{\mu}(T)\}
\end{equation*}
is a Baire space.
\end{lemma}

If $U$ is an invertible measure-preserving transformation on a space $Z$, and $f:Z \to \R$
is any real-valued function, we let
$\avg_m^n(f)$ denote the  Ces\`aro average given by
\begin{equation*}
\avg_m^n(f)(x) :=   \frac{1}{n-m}\sum_{k=m} ^{n-1} f(U^kx).
\end{equation*}
Let $\mu$ be a $U$-invariant measure on $Z$. By the Birkhoff ergodic theorem,
$\mu$ is ergodic if for any $f\in L^1$, $\avg_0^n(f)$ converges in measure to a constant.
In the case that $Z$ is a compact metric space, then we say that $x\in Z$ is {\df generic}
(for $\mu$)  if for every continuous function $f:Z \to \R$, we have that $\avg_0^n(f)(x) \to \mu(f)$ as $n \to \infty$.
An elementary argument gives that if $\mu$ is ergodic if and only if $\mu$-a.e.\ $x \in Z$ is generic for $\mu$.

\begin{proof}[Proof of Lemma \ref{baire-lemma}]
Since $S$ is assumed to be asymptotically entropy expansive, by Lemma \ref{lem:Misiurewicz}
the entropy functional is upper-semicontinuous.   Thus the subset of  joinings of $J_\mu(T,S)$
satisfying the entropy inequality is closed, and hence complete.  We next show that the subset of these joinings
that are ergodic forms a $G_\delta$ subset of this set.

Let $D \subset C(X\times Y) $ be a countable dense collection of continuous functions.
The collection of $\xi$ in $\M_0$ satisfying the condition: for all $f \in D$
and for all $j$, there exists an $n$ such that $\xi (\avg_0^n(f)^2 ) -\xi(f) ^2<1/j$ is  a $G_\delta$ set;
clearly, the condition is that the limit inferior of the variances of the $n$-step C\'esaro averages is 0. This condition
is satisfied if and only if $\avg_0^n(f)$ converges in measure (with respect to $\xi$)
to a constant for each $f \in D$, which holds if and only if $\xi$
is ergodic.

A $G_\delta$ subset of a complete metric space is a Polish space by a theorem
of Alexandrov \cite[Theorem 2.2.1]{Srivastava}; and the Baire category theorem tells us that
every Polish space is a Baire space \cite[Theorem 2.5.5]{Srivastava}.
\end{proof}

\begin{proposition}
\label{enl-open}
Let $(X, \mu, T)$ be  a non-trivial ergodic subshift with invariant measure $\mu$,
with its natural generating partition $\mathcal P$.
Let $S$ be a self-homeomorphism with almost weak specification on a
compact metric space $Y$  that satisfies the small
boundary condition witnessed by a sequence of refining partitions
$(\mathcal Q_{\ell})$, where
$\mathrm{diam}(\mathcal Q_{\ell}) < 1/\ell$.
Let $h_{\mu}(T) < h_{\mathrm{top}}(S)$.
For each $\ell, n \geq 1$, let  $E^{n,\ell}_\mu$ be the set of elements $\xi$ of $\mathcal M_0$ (defined in
Lemma \ref{baire-lemma}) satisfying the following two conditions:
\begin{enumerate}[(i)]
\item
\label{firstcon}
$\mathcal T_X\otimes \sigma(\mathcal Q_\ell)\stackrel{1/n}{\subset}
\left(\bigvee_{i\in\Z}T^{-i}\mathcal P \right)\otimes \mathcal T_Y \mod \xi$;
\item
\label{secondcon}
$\sigma(\mathcal P)\otimes \mathcal T_Y\stackrel{1/n}{\subset} \mathcal T_X
\otimes \bigvee_{i\in\Z}S^{-i}\mathcal Q_\ell \mod \xi$.
\end{enumerate}
Let
$E^n_\mu=\bigcup_{\ell\ge n}E^{n,\ell}_\mu$.  We have the following consequences.
\begin{enumerate}[(I)]
\item
\label{open}
The set $E^{n,\ell}_\mu$ is a relatively open subset of $\mathcal M_0$.
\item
\label{intersection}
If $\xi \in \bigcap_{n\geq 1} E^n_{\mu}$, then $\xi$ is an embedding.
\item
\label{dense}
The set $E^n_{\mu}$ is a dense subset of $\mathcal M_0$.
\end{enumerate}
\end{proposition}

Let us make a few remarks.
First, by \eqref{intersection}, we may call $E^n_\mu$ the collection of
{\df $\boldsymbol{1/n}$-approximate embeddings} of $(X,T,\mu)$ into $(Y,S)$.
Second, the small boundary condition is used to prove \eqref{open}, and
almost weak specification is used to prove the density condition \eqref{dense}.
Finally, most of the hard work will done in verifying \eqref{dense}.

Before we prove the easier parts of Proposition \ref{enl-open}, let us put together
a proof of Theorem \ref{result-gen}.   We need one more lemma in order to obtain fully supported measures on $Y$.

\begin{lemma}
\label{full-support}
Let $(X, \mu, T)$ be  a non-trivial ergodic subshift with invariant measure $\mu$,
and let $S$ be a self-homeomorphism  with almost weak specification 
on a compact metric space $Y$.   Let $h_{\mu}(T) < h_{\mathrm{top}}(S)$.
Let $\M_0'$ be the  set of all $\xi \in \M_0$ such that
$\pi_2^{*} (\xi)$ is fully supported on $Y$.   Then $\M_0'$ is an intersection of
countably many relatively open dense subsets of $\M_0$.
\end{lemma}

We give the proof of Lemma \ref{full-support} in Section \ref{section-up}.

\begin{proof}[Proof of Theorem \ref{result-gen}]
By Lemma \ref{kriegerized}, we may assume that $(X,T,\mu)$ is a non-trivial ergodic subshift.
By Lemma \ref{baire-lemma}, $\mathcal M_0$ is a Baire space, and by  Proposition
\ref{enl-open} \eqref{open} and \eqref{dense},  we have that
$\mathcal E := \bigcap_{n\geq 1} E^n_{\mu}$ is a (non-empty) dense subset
of $\mathcal M_0$; furthermore, by Lemma \ref{full-support}, $\M_0' \cap \mathcal E$
is also a (non-empty) dense subset of $\mathcal M_0$.
It follows from Proposition \ref{enl-open} \eqref{intersection} that the joinings of
$\mathcal E$ are also embeddings.
\end{proof}

\begin{proof}[Proof of Proposition  \ref{enl-open} \eqref{open}]
If $\xi\in E^n_\mu$,
then each element, $X\times A$ with $A\in \sigma(\mathcal Q_\ell)$
agrees with an element $B\times Y$ with
$B\in\bigvee_{i\in\Z}T^{-i}\mathcal P$ up to a
symmetric difference of measure strictly less than $1/n$; moreover,
since $\bigvee_{i\in\Z}T^{-i}\mathcal P$ is the limit of
$\bigvee_{|i|\le m}T^{-i}\mathcal P$, there is a  finite  $m$ such that for all
 $A$ in the finite set $\sigma(\mathcal{Q}_{\ell})$, there is a corresponding clopen set
$B_A\in\bigvee_{i=-m}^mT^{-i}\mathcal P$ such that
\begin{equation}
\label{per}
\xi((X\times A)\bigtriangleup (B_A\times Y))<1/n.
\end{equation}

  Note that for  a fixed $C \in X \times Y$, the map $\xi' \mapsto \xi'(C)$ is continuous at all points $\xi'$ with $\xi'(\partial C) = 0$.  Thus    inequality \eqref{per} persists for all sufficiently small perturbations of $\xi$, if  $\xi(X \times \partial A) =\pi_2^*(\xi)(\partial A)=0$ and $\xi(\partial B_A \times Y) =\pi_1^*(\xi)(\partial B_A)=0$; the latter property holds since $B_A$ is a clopen set.
  For the former property, note that for all $\xi' \in \mathcal M_0$ we have  $\xi'(X\times \partial A)=0$, since $\pi_2^*(\xi')$
is an $S$-invariant measure and by the  small boundary property, we have
$\pi_2^*(\xi')(\partial A) = 0$.

The openness of the second condition
is proved similarly (using the fact that each element of $\bigvee_{|i|\le k}S^{-i}\mathcal Q_\ell$ has
boundary of measure 0 for each $k$ and $\ell$ and every invariant measure).
\end{proof}

\begin{proof}[Proof of Proposition \ref{enl-open} \eqref{intersection}]
The proof follows from Proposition \ref{rue}.
\end{proof}

The remainder of this paper is dedicated to proving Proposition \ref{enl-open} \eqref{dense}
and Lemma \ref{full-support}.
We choose a $\xi\in \mathcal M_0$ and a weak$^*$-neighbourhood of $\xi$. We
show the denseness in several stages.    Firstly it will be convenient to perturb
$\xi$ so that $h_{\pi_2^*(\xi)}(S)$ strictly
exceeds $h_{\mu}(T)$ (it is already at  least $h_{\mu}(T)$ by assumption).
We do this in Section \ref{section-up}.
Similar techniques will also be used to prove Lemma \ref{full-support}.
Secondly, using the fact that $h_{\pi_2^*(\xi)}(S)>h_{\mu}(T)$, we build a mapping from blocks of $X$
into separated orbit segments in $Y$.   In a third stage, we put this together with some
marker blocks in $Y$, verify that the weak$^*$-closeness is satisfied, the entropy condition
still holds and show that the almost embedding property is satisfied.

  The crucial property that permits us to make our constructions is almost weak specification, and in the next section we prove a proposition which serves as a basis for several of our constructions.

\section{Specification}

Let $(X, \mu, T)$ be a non-trivial ergodic subshift with invariant measure $\mu$.
Let $S$ be a self-homeomorphism of a  compact metric space $(Y, \metric)$ satisfying almost weak specification.  The principal task in this paper will be constructing $\mu$-joinings on $X\times Y$.
In fact, we will make three distinct constructions at different points of the
proof (to raise entropy, ensure full support of the $Y$-marginal,
and to create an approximate factor map). The method has a great deal of flexibility.   We rely crucially on the
specification properties of $Y$, which allows us to concatenate segments of
$Y$ orbit and interpolate them.

In all three cases, we are looking for a perturbation of an existing ergodic
$\mu$-joining.   Another common feature is that we leave the $X$ part of the joining
alone, and just modify the $Y$ part.   Finally, all three constructions make use
of external randomization: we take a product of what we are starting from (in two of
the cases, a previously existing joining; and in the third, the measure $\mu$
on $X$) with a mixing process. The new joining is obtained as a factor of this
product.    See Figure \ref{fig:one} for an illustration.

\begin{figure}
\includegraphics[width=4in]{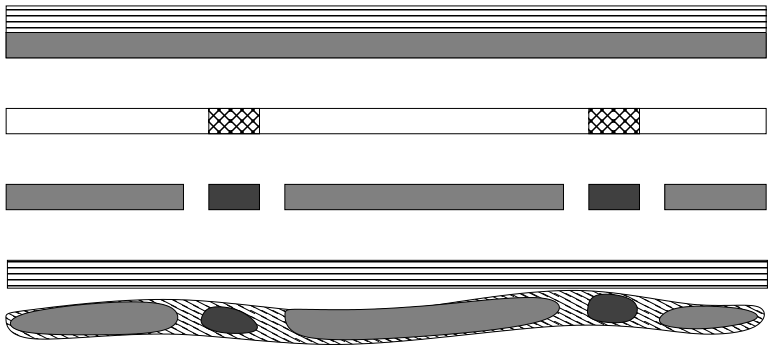}
\caption{The top part of the diagram represents a pair $(x,y)\in X\times Y$, taken from the
joining $\xi$. The next row is $z$, the independent randomness. The third row
depicts the $Y$ pseudo-orbit that we wish to shadow, which we can think of as the output of a function $\psi(x,y,z)$; and the bottom row is the
pair $(x,\tilde y)$ obtained by combining the original $x$ with the true $Y$ orbit obtained from
the specification.}\label{fig:one}

\end{figure}

The starting point of all three constructions is the following proposition.
 Given two  sequences of points
$u = (y_i)$, $v = (y_i')$ of $Y$, and $r >0$,
we say that  {\df $\boldsymbol{u}$ $\boldsymbol{r}$-shadows $\boldsymbol{v}$
in $\boldsymbol{I}$}   if $\metric(y_i, y_i') \leq r$ for all $i \in I$.  For convenience, we  will make use of a special symbol  $\nullcon \not \in Y$ (with $\nullcon$  standing for `vacuous').

\begin{proposition}[Interpolation]
\label{interpolation}
Let $S$ be a self-homeomorphism with almost weak specification
 on a  compact metric space  $(Y, \metric)$.
Let $(W,  \upsilon, U)$ be an ergodic measure-preserving system.    Let $\nullcon \not \in Y$.  Let $r>0$, and $L_{r}$ be a corresponding gap function.      Let $N >0$.  Let $\psi:W \to (Y \cup {\nullcon}) ^{\Z}$ be a measurable function such that for $\upsilon$ almost all $w \in W$:
\begin{enumerate}
\item
\label{orbit}
for all $i \in \Z$,  if $\psi(w)_i \in Y$, then either $S(\psi(w)_i) = \psi(w)_{i+1}$ or $\psi(w)_{i+1} = \nullcon$;
\item
\label{block-size}
  the sequence $\psi(w)$ never contains more than  $N$ consecutive elements of $Y$;
 \item
  \label{spec-garbage}
  all occurrences of $\nullcon$ always appear in blocks of size at least $L_r(N)$; that is, if $\psi(w)_i = \nullcon$, then there are integers $a \leq i \leq b$ such that $b-a \geq L_r(N)$ and for all $ a \leq j \leq b$ we have $\psi(w)_j = \nullcon$.
\end{enumerate}
  Then there exists an ergodic measure $\varrho$ on $W \times Y$ that is preserved by the product transformation $U \times S$, has $\upsilon$ as its projection on $W$, and has the property that  for $\varrho$ almost all $(w,y)$  the sequence $(S^n y)_{n \in \Z}$ $r$-shadows $\psi(w)$  in the set $I$ of all integers $k$ such that  $\psi(w)_k \in Y$.
\end{proposition}

\begin{proof}%[Proof of Proposition \ref{interpolation}]
Consider a point $\mathbf{w} \in W$ that is $\upsilon$-generic.   Almost weak specification and the assumptions on $\psi$ imply that there exists a $\mathbf{{\tilde {y}}} \in Y$ such that the sequence $(S^n \mathbf{{\tilde {y}}})_{n \in \Z}$ $r$-shadows $\psi(\mathbf{w})$ in the set $I$.    Thus by standard arguments (see the proof of the Krylov-Bogolyubov theorem
\cite[Page 97]{Glasner}) there exists a $(U \times S)$-invariant measure $\tilde \varrho$ on $W \times Y$ such that for $\tilde \varrho$ almost all $(w,y)$ we have that the sequence $(S^n y)_{n \in \Z}$ $r$-shadows $\psi(w)$  in the set $I$.  Let $\varrho$ be an ergodic component of $\tilde \varrho$.  Since $\upsilon$ is ergodic and the original point $\mathbf{w}$ was chosen to be generic, the projection of  $\varrho$ on $\Omega$ is $\upsilon$.  Thus  $\varrho$ has the desired properties.
\end{proof}

\section{Upping the entropy}
\label{section-up}

\begin{lemma}[Perturbation to increase marginal entropy]
\label{up}
    Let $(X,  \mu, T)$ be a non-trivial ergodic subshift with invariant measure
    $\mu$. Let $S$ be a
    self-homeomorphism with almost weak specification on a compact metric space $Y$. 
    Let  $h_\mu(T) < h_\mathrm{top}(S)$.
     Let $\xi\in \mathcal M_0$ 
    and let $V$ be a weak$^*$-neighbourhood of $\xi$ in $\mathcal M_0$. Then there exists a $\xi'\in V$
    such that $h_{\pi_2^*(\xi')}(S)>h_{\mu}(T)$.
\end{lemma}

To prove Lemma \ref{up}, we alter the measure  $\xi$ by the following procedure.   Given an ergodic measure $\lambda$ on $Y$ with $h_{\lambda}(S) > h_{\mu}(T)$,
a  $\xi$-generic  point $(x,y^1)$, and a $\lambda$-generic point $y^2$ that is chosen independently of $(x,y^1)$,
we produce another
point $(x,y^3)$, where $y^3$ is obtained from  $y^1$ by splicing segments
of $y^1$ with segments of $y^2$.   This is made possible by Proposition \ref{interpolation} and a careful choice of parameters  that are used to define
perturbation of $\xi$ so that we have an entropy increase, but
remain within a prescribed  weak$^*$ neighborhood.   The proof may appear somewhat technical
because an arbitrary generating partition for $Y$ may
not have any relation to the specification properties of $Y$, and we also need to construct the
perturbation so that we can verify the
required properties.

We make use of the d-bar distance that was introduced by Ornstein \cite{MR0447525}.  We first define the d-bar distance between two processes, and then extend the definition to general measure-preserving systems.
By a {\df joining} of two probability measure spaces each endowed with a transformation,
we mean a probability measure on the product space that is
invariant with respect to the product transformation and has coordinate projections equal to the original measures.
Let $(\mathrm{\Sigma},\theta)$ be the full-shift on a finite number of symbols,  endowed with    invariant measures $\mu_1$ and  $\mu_2$.   Let
\begin{equation}
\metricbar(\mu_1, \mu_2) := \inf_{\xi \in J(\mu_1, \mu_2)} \int \mathbf{1}[\mathrm{\omega}_0 \not= \mathrm{\omega}'_0] d\xi(\mathrm{\omega},\mathrm{\omega}'),
\end{equation}
where $J(\mu_1, \mu_2)$ denotes the set  joinings of $(\mathrm{\Sigma},\mu_1,\theta)$ and $(\mathrm{\Sigma}, \mu_2, \theta)$.

More generally, let $U$ be an invertible measure-preserving transformation  of  $(Z, \nu)$, and let
 $\mathcal P$ be a measurable partition with cardinality $| {\mathcal P}|$.
  For each $z \in Z$, let
$\mathcal P(z) \in \mathcal P$ denote the part that contains $z$.
The {\df $\mathcal {P}$-name} is the  function $\mathcal{\bar{P}}:Z \to \ns{1, \ldots, |{\mathcal P}|}^{\Z}$  given by
$\mathcal{\bar{P}}(z)_i := \mathcal P(U^{i}z)$.   Thus
$\nu^{\mathcal P}:=\nu \circ \mathcal{\bar{P}}^{-1}$ is a shift-invariant measure on
the full-shift of $|{\mathcal P}|$ symbols.
We may define the d-bar distance between two invertible measure-preserving transformations
$(Z_1, \nu_1, U_1)$ and $(Z_2, \nu_2, U_2)$
endowed with finite  measurable partitions $\mathcal P$ and $\mathcal Q$ by considering $\nu_1^{\mathcal P}$ and $\nu_2^{\mathcal P}$ to be shift-invariant
measures on the full-shift of $\max(|{\mathcal P}|,|{\mathcal Q}|)$ symbols, and setting
\begin{equation}
\metricbar(\nu_1, \mathcal P ;   \nu_2, \mathcal Q) := \metricbar(\nu_1^{\mathcal P}, \nu_2^{\mathcal Q}).
\end{equation}

We define $\mathcal H:(0,1) \to \R$ via
\begin{equation}
\label{Phi}
 \mathcal H(\eta):=-\eta\log\eta-(1-\eta)\log(1-\eta).
\end{equation}

\begin{lemma}
\label{phi-en}
Let  $(X, \mu, T)$ and $(Z, \nu, U)$ be two invertible ergodic measure preserving transformations
endowed with measurable partitions $\mathcal P$ and $\mathcal Q$.  Let $\eta >0$.  If
$\metricbar(\mu, \mathcal P;  \nu, \mathcal Q) < \eta$ and $| \mathcal P| \geq |\mathcal Q|$, then
\[
|h_{\mu}(T,\mathcal P)- h_{\nu}(U, \mathcal Q)| \leq \mathcal H(\eta) + \eta\log|\mathcal P|.
\]
\end{lemma}

\begin{proof}
See Rudolph's book \cite[Theorem 7.9]{Rud}
\end{proof}

\begin{proof}[Proof of Lemma \ref{up}]
 If $h_{\pi_2^*(\xi)}(S)>h_{\mu}(T)$, then we are done, so we assume that $h_{\pi_2^*(\xi)}(S)=h_\mu(T)$
 in what follows.

Choose an ergodic measure $\lambda$ on $Y$ satisfying
$h_{\lambda}(S)>h_{{\pi_2^*(\xi)}}(S)=h_{\mu}(T)$.
Choose  $\epsilon \in (0,1)$ so that $V$ contains the ball $B(\xi, 3\epsilon) \cap \M_0$.
Let $\delta=\frac{1}{4}(h_{\lambda}(S)-h_{{\pi_2^*(\xi)}}(S))$.

By a variation of a standard argument (see for example \cite[Lemma 3.2]{Bowen}), we may choose
a partition $\mathcal P=\ns{B_1, \ldots, B_n}$ of $Y$ such that
${\pi_2^*(\xi)}(\partial \mathcal P)=\lambda(\partial \mathcal P)=0$ and
$h_{\lambda}(S,\mathcal P)>h_{\lambda}(S)-\epsilon\delta$ and $h_{{\pi_2^*(\xi)}}(S,\mathcal P)>
h_{{\pi_2^*(\xi)}}(S)-\epsilon\delta$.
Let $\mathcal H$ be as in \eqref{Phi} and choose $\gamma<\epsilon$ so that $ \mathcal H(2\gamma)
+ 2\gamma\log(|\mathcal P| +1) <\epsilon\delta$
and $2\gamma h_\mathrm{top}(S)<\epsilon\delta$.
Let $r<\epsilon$ be chosen so that ${\pi_2^*(\xi)}(\partial_r\mathcal P)<\gamma$ and
$\lambda(\partial_r\mathcal P)<\gamma$.

We now build a   mixing process taking values  in $\{0,1,2\}$.   Choose $N$ sufficiently
large so  that the gap function satisfies
\[
L_r(N) < k_0:= \lfloor (\gamma/2) N \rfloor;
\]
in addition,   we require that
\begin{equation}
\label{time-ratio}
\frac{k_1}{2k_0 +k_1+k_2} > 1- \e - \gamma \ \text{ and } \ \frac{k_2}{2k_0+k_1+k_2} > \e - \gamma,
\end{equation}
where $k_1:=\lfloor (1-\epsilon-\gamma/2)N\rfloor$ and
$k_2:=\lfloor (\epsilon-\gamma/2)N\rfloor$.    The process consists of
concatenations of the blocks
\[1^{k_1}0^{k_0}2^{k_2}0^{k_0} \text{ and } 1^{k_1+1}0^{k_0}2^{k_2}0^{k_0},\]
where the blocks are placed independently with equal probabilities of using
the longer and the shorter block.   (We need to use  values $k_1$ and $k_1 +1$ in order to make $\zeta$ mixing.)
Write $\zeta$ for the measure on $R:=\{0,1,2\}^{\Z}$.

It will be convenient to have notation that will allow us to distinguish various copies of $Y$
in certain product spaces.  Let $Y^{(i)} =Y$ for $i=1,2,3$. Let $\sigma$ be the shift on
$R$ and consider the transformation
$\tau:=\sigma\times(T\times S)\times S$ acting on
$\Omega:=R\times (X\times Y^{(1)})\times Y^{(2)}$, preserving the
(not necessarily ergodic) measure $\zeta\times \xi\times \lambda$.
Define a new partition $\mathcal Q$ on the product by
 $D_i:=\{(z,y^{(1)},y^{(2)})\colon z_0\in\{1,2\}\text{ and }
y^{(z_0)}\in B_i\}$ and $D_0:=\{(z,y^{(1)},y^{(2)})\colon z_0=0\}$.
Note that $|\mathcal P|+1 = |\mathcal Q|$.
The idea is knowing which element of
$\mathcal Q$ a point $(z,y^{(1)},y^{(2)})$ lies in tells us which element of $\mathcal P$ the $z_0$'th copy
of $y$ lives in if $z_0$ is 1 or 2; or else tells us that $z_0=0$. Let $\mathcal R$ denote the
$\sigma$-algebra of $\Omega$ giving information about the $R$ coordinates.

Using the fact that measures ${\pi_2^*(\xi)}$ and $\lambda$ are independently joined in the  product measure $\zeta \times \xi \times \lambda$, we have   using \eqref{time-ratio} and the definition of $\mathcal R$ that
\begin{align*}
  % \label{cross-up}
  \lim_{M\to\infty}\frac1{M}H_{\zeta\times\xi\times\lambda}&
  \left(\bigvee_{i=0}^{M-1} \tau^{-i}\mathcal Q\right)
  \ge\lim_{M\to\infty}\frac1MH_{\zeta\times\xi\times\lambda}
  \left(\bigvee_{i=0}^{M-1}
    \tau^{-i}\mathcal Q\bigg\vert\mathcal R\right)   \nonumber\\
  &\ge (1-\epsilon-\gamma) h_{\pi_2^*(\xi)}(S,\mathcal
  P)+(\epsilon-\gamma)h_{\lambda}(S,\mathcal P),
\end{align*}
since for each coordinate where $z_i=1$, the expected amount of information revealed is at least $h_{\pi_2^*({\xi})}(S, \mathcal P)$ and for each coordinate where $z_i=2$, the expected amount of information is at least $h_\lambda(S, \mathcal Q)$.

By affineness of entropy and the ergodic decomposition theorem,  we can take an
ergodic component, $\iota$, of $\zeta\times \xi\times \lambda$
such that
\begin{equation}
\label{cross-up}
h_{\iota}(\tau,\mathcal Q) \geq (1-\epsilon-\gamma)h_{\pi_2^*(\xi)}(S,\mathcal P)+
(\epsilon-\gamma)h_{\lambda}(S,\mathcal P).
\end{equation}
Since $\zeta$ is mixing, both $\zeta \times {\pi_2^*(\xi)}$ and  $\zeta \times \lambda$ are ergodic.
Thus the projections of $\iota$ on $R \times Y^{(1)}$ and $R \times Y^{(2)}$
are (still) $\zeta \times {\pi_2^*(\xi)}$ and $\zeta \times \lambda$,
respectively.

By Proposition \ref{interpolation},
there exists a $(\tau \times S)$-invariant measure $\iota'$ on
$ \Omega':=\Omega\times Y^{(3)}$ such that for $\iota'$-a.e.\ point  $\omega'=(z,(x,y^{(1)}),y^{(2)},y^{(3)})$, we have
\begin{equation}
\label{transtoergodic}
\metric(S^ny^{(3)},{S}^n y^{(z_n)})
\leq
r  \text{ for all $n$ such that }z_n\in\{1,2\};
\end{equation}
furthermore, the projection of $\iota'$ on $\Omega$ is $\iota$.

\begin{comment}
\note{R11}Almost weak specification gives that for an $\iota$-generic
point $ \boldsymbol{\omega} =  (\mathbf{z},(\mathbf{x},\mathbf{y}^{(1)}),
\mathbf{y}^{(2)}) \in \Omega$,
    we can build a $\mathbf{\tilde{y}}$ such that
\noteend{}
\begin{equation}
\label{buildapoint}
\metric(S^n \mathbf{{\tilde {y}}}  ,{S}^n \mathbf{y}^{({\mathbf{z}_n)}})<r
\text{ for all $n$ such that }\mathbf{z}_n\in\{1,2\}.
\end{equation}

Thus by standard arguments (see the proof of the Krylov-Bogolyubov theorem
\cite[Page 97]{Glasner}) and \eqref{buildapoint}, there exists a $(\tau \times S)$-invariant measure $\tilde \iota$ on
$ \Omega':=\Omega\times Y^{(3)}$ such that for $\tilde \iota$-a.e.\ point  $(z,(x,y^{(1)}),y^{(2)},y^{(3)})$, we have
\note{R13}
\begin{equation}
\label{transtoergodic}
\metric(S^ny^{(3)},{S}^n y^{(z_n)})
\leq
r  \text{ for all $n$ such that }z_n\in\{1,2\}.
\end{equation}
\noteend{}
Furthermore, by taking an ergodic component $\iota'$ of $\tilde \iota$ we
may obtain an ergodic measure such that  \eqref{transtoergodic} holds for $ \iota'$-a.e.\ point.
Since $\iota$ is ergodic and the original point $\boldsymbol{\omega}$ was chosen
to be generic,  the projections of
$\iota'$ and $\tilde \iota$  on   $\Omega$ are $\iota$.
\end{comment}

We now define $\tilde \pi(z,(x,y^{(1)}),y^{(2)},y^{(3)}):=(x,y^{(3)})$.
Note that by construction,
$y^{(3)}$ is within $r$ of $y^{(1)}$ whenever $z_0=1$.
Let $\xi':=\tilde\pi^*(\iota')$. Clearly, $\xi'$ is ergodic.

We now give estimates on the weak$^*$ distance between $\xi$ and $\xi'$ and
on the entropy $h_{\pi_2^*(\xi')}(S)$.    Let $f \in \mathrm{Lip}_1(X \times Y)$ (recalling that $f$ takes values in
$[0,1]$ by definition).
Extend $f$ to functions $f^1, f^3$ on $\Omega'$  by setting
$f^1(z,(x,y^{(1)}),y^{(2)},y^{(3)})= f(x, y^{(1)})$ and
$f^3(z,(x,y^{(1)}),y^{(2)},y^{(3)})= f(x,y^{(3)})$.   Also let
\[
D:= \ns{(z,(x,y^{(1)}),y^{(2)},y^{(3)}) \in  \Omega':  z_0 = 1}.
\]
We then have
\begin{align*}
| \xi(f) - \xi'(f)| &= |  \iota' (f^1) -  \iota' (f^3)|  \\
& \leq \iota'(  \mathbf{1}_D| f^1 - f^3|) +  \iota'(  \mathbf{1}_{D^c}| f^1 - f^3|).
\end{align*}
Thus by \eqref{transtoergodic},  we have
\begin{equation}
\label{raw-weakstar}
\metricstar(\xi, \xi') \leq  r + (\epsilon+\gamma) <3\e.
\end{equation}

As for entropy, we estimate it by using \eqref{cross-up} and the $\metricbar$ distance
between $\mathcal P$-name of $y^{(3)}$ and the
$\mathcal Q$-name of $(z,(x,y^{(1)}),y^{(2)},y^{(3)})$.

For $\omega' \in \Omega'$, if $z_0=1$ and $y^{(1)}\not\in \partial_r\mathcal P$,
then $y^{(1)}$ and $y^{(3)}$ lie in the same element of $\mathcal P$.
Similarly if $z_0=2$ and
$y^{(2)}\not\in  \partial_r\mathcal P$, then $y^{(2)}$ and $y^{(3)}$ lie in the same
element of $\mathcal P$.
In particular, if $z_0=1$ and $y^{(1)}\not\in  \partial_r\mathcal P$ or $z_0=2$ and
$y^{(2)}\not\in \partial_r\mathcal P$, then $y^{(3)}$ lies in $B_j$
if and only if $(z,(x,y^{(1)}),y^{(2)},y^{(3)})$ lies in $D_j$.

On the other hand,  since the projection of $\iota'$ on $\Omega$ is $\iota$ and
the projections of $\iota$ on $R \times Y^{(1)}$ and $R \times Y^{(2)}$ are $\zeta \times {\pi_2^*(\xi)}$
and $\zeta \times \lambda$, we have
\begin{align*}
\iota'\{\omega'\in \Omega'\colon y^{(1)}\in
\partial_r\mathcal P\text{ and }z_0=1\}
&= {\pi_2^*(\xi)}(\partial_r\mathcal P)    \cdot \zeta\{z\colon z_0=1\} \text{; and}\\
\iota'\{\omega'\in \Omega'\colon y^{(2)}\in
\partial_r\mathcal P\text{ and }z_0=2\}
&=  \lambda(
\partial_r\mathcal P)  \cdot \zeta\{z\colon z_0=2\}.
\end{align*}

Note that $\pi_2^*(\xi')$ is the projection of ${\iota'}$ on $Y^{(3)}$.    Let $\mathcal P'$ be
the partition of $\Omega'$ whose parts are given by
\[
\ns{\omega' \in \Omega': y^{(3)} \in B_i} \text{ for } 1 \leq i \leq n.
\]
Let $\mathcal Q'$ be the partition of $\Omega'$ whose parts are given by
\[
\ns{\omega' \in \Omega': (z,(x,y^{(1)}),y^{(2)}) \in D_i} \text{ for }  0 \leq i \leq n.
\]
It follows that the d-bar distance $\metricbar({\pi_2^*(\xi')}, \mathcal P' \,;\, {\iota'}, \mathcal Q')$
is bounded above by
\[
{\pi_2^*(\xi)}(\partial_r\mathcal P)    \cdot \zeta\{z\colon z_0=1\} +  \lambda(
    \partial_r \mathcal P)  \cdot \zeta\{z\colon z_0=2\} + \zeta\ns{z: z_0 =0}<2\gamma.
\]

By Lemma \ref{phi-en}, we have
\[
|h_{{\pi_2^*(\xi')}}(S,\mathcal P)-h_{\iota'}(\tau \times S,\mathcal Q)| \leq  \mathcal H(2\gamma) +
2\gamma\log(|\mathcal P| +1)
<\epsilon\delta.
\]
Thus by \eqref{cross-up} and the fact that $h_{\iota}(\tau,\mathcal Q) =
h_{ {\iota'}}(\tau \times S, \mathcal Q')$, we have
\begin{align*}
&h_{{\pi_2^*(\xi')}}(S)\\ &\geq  h_{{\pi_2^*(\xi')}}(S,\mathcal P)  \\
&\geq  (1-\epsilon-\gamma)h_{\pi_2^*(\xi)}(S,\mathcal P)+
(\epsilon-\gamma)h_{\lambda}(S,\mathcal P) - \epsilon\delta\\
&\ge (1-\epsilon)h_{\pi_2^*(\xi)}(S)-\epsilon\delta + \epsilon (h_{\pi_2^*(\xi)}(S)+4\delta)-
\epsilon\delta-2\gamma h_\mathrm{top}(S)-\epsilon\delta\\
&>h_{\pi_2^*(\xi)}(S).
\qedhere
\end{align*}
\end{proof}

The proof of Lemma \ref{full-support} uses similar ideas, but is made easier
by the fact that we can use Lemma \ref{up}.

\begin{proof}[Proof of Lemma \ref{full-support}]
Let $(V_n)_{n\in\N}$ be a countable collection of open sets forming a neighbourhood basis
for $Y$. Let $C_n:=\{\xi\in\M_0\colon \pi_2^*(\xi)(V_n)>0\}$. By standard properties of the weak$^*$-distance,
$C_n$ is a relatively open subset of $\M_0$ in the weak$^*$ topology. It suffices to show that
$C_n$ is dense in $\M_0$ for each $n$.

Fix $n\in \N$ and $\epsilon>0$. Let $\xi_0\in \M_0$. By Lemma \ref{up}, there exists $\xi_1\in\M_0$
such that $h_{\pi_2^*(\xi_1)}(S)>h_\mu(T)$ and $\metricstar(\xi_0,\xi_1)<\epsilon/2$.
%Let $\nu=\pi_2^*\xi_1$.
Let  $\mathcal Q$ be a partition such that $\pi_2^*(\xi_1)(\partial Q) = 0$ and
$h_{\pi_2^*(\xi_1)}(S,\mathcal Q)>h_\mu(T)$.
Let $\delta<\epsilon/4$ be such that the following inequality holds:
\[  \mathcal H(2\delta) + 2\delta\log|\mathcal Q|<h_{\pi_2^*(\xi_1)}(S,\mathcal Q)-h_\mu(T),\]
where $\mathcal H$ is defined in \eqref{Phi}.

Let $r<\epsilon/4$ be chosen such that ${\pi_2^*(\xi_1)}(\partial_r\mathcal Q)<\delta$  and
$B(\mathbf{y},r)\subset V_n$
for some fixed $\mathbf{y} \in Y$.   We then choose $N$ such that $L_r(N)< M:= \lfloor\delta N/2 \rfloor$.
Let $\zeta$ be a mixing measure on $Z=\{0,1,2\}^\Z$ supported on concatenations of the blocks
$1^{N-2M}\,0^{M}\,2\,0^{M}$ and
$1^{N+1-2M}\,0^{M}\,2\,0^{M}$. As in the proof of Lemma \ref{up}, using Proposition \ref{interpolation},
we start with the (ergodic) product measure $\xi_1\times\zeta$ on $(X\times Y)\times Z$
and obtain an ergodic measure $\iota$ on $(X\times Y)\times Z\times Y$
whose marginal on the first pair of coordinates is $\xi_1$ and which satisfies
$\metric(y',y) \leq r$ for $\iota$-almost every $((x,y),z,y')$ such that $z_0=1$
and $\metric(y',\mathbf{y}) \leq r$ for $\iota$-almost every $((x,y),z,y')$ such that $z_0=2$.
Projecting $\iota$ onto the initial $X$ coordinate and the final $Y$ coordinate gives a measure $\xi'$
that can be checked using the same arguments as in Lemma \ref{up} to satisfy
$\metricstar(\xi_1,\xi')<\epsilon/2$ and
$\metricbar({\pi_2^*(\xi_1)},\mathcal Q \,;\,\pi_2^*(\xi'),\mathcal Q)<2\delta$.
This is sufficient to ensure that $h_{\pi_2^*(\xi')}(S,\mathcal Q)>h_\mu(T)$, so that $\xi'\in \M_0$.
Finally, $\pi_2^*(\xi')(B(\mathbf{y},r))>1/(N+1)$, so that $\xi'\in C_n$.
\end{proof}

\section{Marriage via Brin--Katok}
\label{section-marr}

Let $(X, \mu, T)$ be  a non-trivial ergodic subshift with invariant measure $\mu$.
Let $S$ be a self-homeomorphism on a
compact metric space $Y$.   Given a $\mu$-joining, $\xi$,  with $h_{\pi_2^*(\xi)}(S)>h_\mu(T)$,
we will define an injective map from blocks of $X$  to orbit segments of $Y$, with certain
properties that will be useful in constructing perturbations of $\xi$ that are approximate embeddings.

This section, while in spirit closely following the proof of the Krieger generator theorem
given by Burton, Keane and Serafin  \cite{BurKeaSer},
differs in the details because we embed into a general compact metric space, rather
than into a shift space.    Burton, Keane and Serafin  make use of the Shannon-Macmillan-Breiman
theorem and Hall's marriage theorem, which are also important ingredients in the proofs of
the Sinai factor theorem and Ornstein isomorphism theorem given by Keane and Smorodinsky
\cite{keanea, keaneb}.
We substitute  the Shannon-Macmillan-Breiman theorem with a  topological
analogue due to  Brin and Katok \cite{BriKat}.   In order to define {\em markers} in a general
setting that is not necessarily symbolic, we will also make use of a generalization due to
Downarowicz and Weiss \cite{DownWeiss}
of the Ornstein and Weiss \cite{MR1211492} return time formula for entropy.

Let $S$ be a self-homeomorphism on a compact metric space $(Y, \metric)$.
For integers $m < n$ let
\begin{equation*}
\metric_m^n(y,z) := \max_{m \leq j  <  n} \metric(S^jy, S^jz)  \quad \text{for all} \quad y,z\in Y.
\end{equation*}
%Note that the space $Y$ is compact with respect to each of these metrics.
We define the    {\df $\boldsymbol{(\metric_m ^n, \delta)}$-Bowen ball} about a point $y \in Y$  by
\begin{equation*}
B_m ^n(y,\delta) :=  \ns{z \in Y:  \metric_m^n(y,z) < \delta}.
\end{equation*}
For an $S$-invariant measure $\lambda$ on $Y$  we define for $(z, \eta) \in Y \times \R^{+}$
\[
\underline h^\text{BK}_\lambda(z,\eta):=
\liminf_{n\to\infty} -(1/n)\log \lambda(B_0^n(z, \eta))
\]
and
\[
\underline h^\text{DW}_\lambda(z,\eta):=\liminf_{n\to\infty}
(1/n)\log \min\{i>0\colon \metric_0^n(z,S^iz)<\eta\}.
\]
It can be easily checked that $\underline h^\text{BK}_\lambda(z,\eta)$ and
$\underline h^\text{DW}_\lambda(z,\eta)$ are $S$-invariant functions of $z$, so that if $\lambda$ is an
ergodic measure,
there are
monotone
functions
$\underline h^\text{BK}_\lambda(\eta)$
and
$\underline h_\lambda^\text{DW}(\eta)$
such that $\underline h^\text{BK}_\lambda(z,\eta)=\underline h^\text{BK}_\lambda(\eta)$ and
$\underline h^\text{DW}_\lambda(z,\eta)=\underline
h^\text{DW}_\lambda(\eta)$ for $\lambda$-almost every $z\in Y$.

\begin{theorem}[Brin and Katok \cite{BriKat}]\label{thm:BK}
    Let $S$ be a self-homeomorphism of a
    compact metric space $Y$ preserving an ergodic measure $\lambda$ and suppose that
    $h_{\lambda}(S)<\infty$.
    Then for $\lambda$-almost every $z\in Y$, we have
\[
    \lim_{\eta\to 0}\underline h^\mathrm{BK}_\lambda(z,\eta)=
    h_{\lambda}(S).
\]

\end{theorem}

\begin{theorem}[Downarowicz and Weiss \cite{DownWeiss}]
\label{Downarowicz}
\sloppy
    Suppose that  $S$ is a self-homeomorphism of a
    compact metric space $Y$ preserving an ergodic measure $\lambda$ and suppose that
    $h_{\lambda}(S)<\infty$.
Then for $\lambda$-a.e.~$z\in Y$, we have
\[
\lim_{\eta \to0} \underline h^\mathrm{DW}(z,\eta)
=h_\lambda(S).
\]
\end{theorem}	
\fussy

\begin{lemma}[Marker lemma]
\label{marker-lemma}
Let $S$ be a self-homeomorphism of a compact metric space $Y$
preserving an ergodic measure $\nu$ with $0<h_\nu(S)<\infty$.  Let
$\alpha \in (0,1)$ and let $\eta$ be such that $\underline
h^\mathrm{BK}_{\nu}(4\eta)>0$ and $\underline h^\mathrm{DW}_\nu(4\eta)>0$.
Then for all sufficiently large integers $M$, there is a point
$y_\mathrm{mark}\in Y$ with the following properties:
\begin{enumerate}
\item
  $\metric_0^{2M}(S^iy_\mathrm{mark},S^{6M}y_\mathrm{mark}) \ge 4\eta$ for
  $i\in\{0,\ldots,6M-1\}$;\label{slowret1}
\item
  $\nu(H_1)<\alpha/M$, where $H_1:=B_0^{2M}(y_\mathrm{mark},4\eta)$;
  \label{smallball1}
\item
  $\nu(H_2)<\alpha/M$, where $H_2:=B_0^{2M}(S^{6M}y_\mathrm{mark},4\eta)$.
  \label{smallball2}
\end{enumerate}
\end{lemma}

\begin{proof}
  Property \eqref{slowret1}
  comes from applying Theorem \ref{Downarowicz} to $S^{-1}$.
  Properties \eqref{smallball1} and \eqref{smallball2} follow from
  Theorem \ref{thm:BK}.
\end{proof}

\begin{corollary}[Marker decipherability]\label{decipher}
  Let $Y$, $S$, $M$, $\eta$, $H_1$ and $H_2$ be as in the statement of
  Lemma \ref{marker-lemma}.   Suppose that $S$ satisfies almost weak specification.   Let $0 < r< \eta$.  Let $L>0$, and assume that the   gap function satisfies $L_r(L)<M<L-18M$.  Let $z \in Y$ be a point such that
  \begin{enumerate}[(i)]
  \item
    \label{setHa}
    $S^iz \not\in H_1 \text{ for } 0\le i<L;$
  \item
    \label{setHb}
    $S^iz \not\in H_2 \text{ for } 0\le i<L.$
  \end{enumerate}
    Let $\tilde z
  \in Y$ be a point satisfying the specification conditions:
\begin{enumerate}[(a)]
\item
\label{spec-z}
$\metric_M^{L-10M}(\tilde z,z)<r$
%for $M<i<L-10M$;
\item
\label{spec-m}
$\metric_0^{8M}(S^{L-9M} \tilde z, y_{\mathrm{mark}}) < r$;
%$\metric_{L-9M} ^{L-M}(\tilde z,S^{-(L-9M)}y_\mathrm{mark})<r$
%for $L-9M\le i<L-M$;
\end{enumerate}
such a point exists by almost weak specification.
Then $\tilde z$ satisfies
\[
\metric_0^{8M}(S^i\tilde z,y_\mathrm{mark})\ge 4\eta-r
\text{ for }0\leq i <L-9M.
\]
In particular, if $0\le j\le L-9M$, $w:=  S^j \tilde z$, and $w' \in Y$ with   $\metric(S^iw',S^iw) <\eta$ for all $i \in \Z$, then $j$ may be recovered from $w'$    via 
\[
j=(L-9M)-\min\{i\ge 0\colon
\metric_0^{8M}(S^i w',y_\mathrm{mark})<3\eta -r\}.
\]
\end{corollary}

\begin{proof}
We will check three cases.
	
If $0\le i < L-18M$, then we have
\begin{align*}
  \metric_0^{8M}(S^i\tilde z,y_\mathrm{mark})
  &\ge
  \metric_0^{2M}(S^{i+6M} \tilde z,S^{6M}y_\mathrm{mark}) \\
  & \geq \metric_0^{2M}(S^{i+6M} z,S^{6M}y_\mathrm{mark})- r
 \text{ (by  \eqref{spec-z})} \\
  & \ge 4\eta-r \text{ (by  \eqref{setHb})}.
\end{align*}

If $M\le i< L-11M$, then we have
\begin{align*}
  \metric_0^{8M}(S^i\tilde z,y_\mathrm{mark}) &\ge
  \metric_0^{2M}(S^i\tilde z,y_\mathrm{mark})  \\
  &\ge \metric_0^{2M}(S^i z,y_\mathrm{mark})  - r \ (\text{by } \eqref{spec-z})
  \\
  & \ge 4\eta - r \ (\text{by } \eqref{setHa}).
\end{align*}

	If
	 $L-11M\le i<L-9M$, then we have
%	 $\metric_0^{2M}(S^{i+6M} \tilde z,S^{6M}y_\mathrm{mark})
%	 \geq \metric_0^{2M}(S^{i+6M-(L-9M)}y_\mathrm{mark},S^{6M}y_\mathrm{mark})
%	  - \metric_0^{2M}(S^{i+6M-(L-9M)}y_\mathrm{mark},S^{i+6M}\tilde z) \ge 2r - r =r$
%	  by Lemma \ref{marker-lemma}
%	 \eqref{slowret1} and \eqref{spec-m}, and thus $\metric_0^{8M}(S^i\tilde z,y_\mathrm{mark})
%	 \ge
%	\metric_0^{2M}(S^{i+6M} \tilde z,S^{6M}y_\mathrm{mark}) \ge r$.
\begin{align*}
&\metric_0^{8M}(S^i\tilde z,y_\mathrm{mark})  \ge
\metric_0^{2M}(S^{i+6M} \tilde z,S^{6M}y_\mathrm{mark}) \ge \\
&
\metric_0^{2M}(S^{i+15M-L}y_\mathrm{mark},S^{6M}y_\mathrm{mark})
- \metric_0^{2M}(S^{i+15M-L}y_\mathrm{mark},S^{i+6M}\tilde z)\\
&\ge 4\eta-r \text{ (by  Lemma \ref{marker-lemma}
	 \eqref{slowret1}) and \eqref{spec-m}}.
\qedhere
\end{align*}
\end{proof}

The following lemma involves the choice of many constants (which are necessary for our
proof of Proposition \ref{enl-open} \eqref{dense}),  but it follows easily from Theorem
\ref{thm:BK} and the ergodic theorem.
\begin{comment}
  In what follows in many instances we will have
expressions and functions that are formally only defined for integer values.
For  Bowen balls, and nonnegative real numbers $r$, we write $B(y,r,\eta) =
B(y, \lfloor r \rfloor, \eta)$ for all $y \in Y$ and $\eta \geq 0$; we will
adopt this convention in other similar circumstances.
\end{comment}
We will refer to Bowen balls in the product space $X\times Y$, where $X$ is a non-trivial subshift equipped  with the usual metric $\metric_X$, and $(Y, \metric_Y)$ is a compact metric space endowed with a self-homeomorphism.
For the definition of these, we will take the metric to be
\begin{equation}
\label{productmetric}
\metric_{X\times Y}((x,y),(x',y')):=
\max(\metric_X(x,x'),\metric_Y(y,y')).
\end{equation}
\begin{comment}
We introduce some notation that will be used repeatedly, analogous to
a cylinder set in a symbolic space.
\begin{align*}
d_a^b(y,y')&:=\max\{\metric(S^ny,S^ny')\colon a\le n<b\};\\
B_a^b(y,\eta)&:=\{y'\colon d_a^b(y,y')<\eta\}.
\end{align*}
\end{comment}

\begin{lemma}
\label{choose-N}
Let $(X, \mu, T)$ be a  non-trivial ergodic subshift with invariant measure
$\mu$ and natural generating partition $\mathcal P$.    Let $S$ be a self-homeomorphism
with almost weak specification on a compact metric space $Y$  that satisfies the small
boundary condition witnessed by a sequence of refining partitions
$(\mathcal Q_{\ell})$, where
$\mathrm{diam}(\mathcal Q_{\ell}) < 1/\ell$.
Let $\xi$ be a $\mu$-joining with
%$\nu:=\pi_2^*\xi $ and
$h_{{\pi_2^*(\xi)}}(S) > h_{\mu}(T)$.    Let $\e >0$.

Set
 \begin{equation}
 \label{def-cap-delta}
  \Delta :=  (h_{{\pi_2^*(\xi)}}(S) - h_\mu(T)  )/10.
  \end{equation}
Let $0<\delta <\epsilon/80$ satisfy  the inequalities:
\begin{align}
\label{en-cap-delta}
\delta\big(1+h_{\xi}(T \times S)\big)&<\Delta;\\
\label{part-cap-delta}
16\delta(1+\log|\mathcal P|)&<\epsilon\Delta;\\
%
%\label{badstuffentropybound}
%\mathcal H(16\delta)
%-(2\delta)\log(2\delta)-(1-2\delta)\log(1-2\delta)
%&<\epsilon\Delta; \\
%
\label{stupid}
4(1-\delta)(1 -15\delta) & \geq 3.
\end{align}
By Theorems \ref{thm:BK} and \ref{Downarowicz}
, let $\eta<\fracc{\epsilon}{12}$ be chosen so that
\begin{align}
\label{BK-bg}
\underline h^\mathrm{BK}_{\pi_2^*(\xi)}(2\eta,S) &>h_{\pi_2^*(\xi)}(S)-\delta,\\
\label{BK-DD}
\underline h^\mathrm{BK}_{\xi}(2\eta, T\times S) 
&>h_{\xi}(T \times S)-\delta, \text{ and} \\
\label{condition-marker}
h^\mathrm{BK}_{\nu}(4\eta)>0
&\text{ and } \  \underline h^\mathrm{DW}_\nu(4\eta)>0.
\end{align}
Thus by \eqref{condition-marker}, $\eta$ satisfies the conditions of   Lemma \ref{marker-lemma}.
Let $\ell>10/\eta$ so that $\mathrm{diam}(\mathcal Q_\ell)<\eta/10$.
%Let $\Phi$ be defined as in \eqref{Phi} and let  $\gamma<1/e$ be such that
%\begin{equation}
%\Phi(\gamma)+\gamma(1+\log|\mathcal Q_\ell|)<\delta.
%\end{equation}
Let $r<\eta/10$ be such that
\begin{equation}
\label{boundarychosen}
{\pi_2^*(\xi)}(\partial_r\mathcal Q_\ell) < \delta.
\end{equation}
Set
\begin{equation}
\label{theM}
M:=\lfloor \delta N/11 \rfloor.
\end{equation}

Then then  following conditions hold for all sufficiently large $N$:
\begin{enumerate}[(a)]
\item
\label{easycond-b}
$e^{N\Delta}>2;$
\item
\label{easycond}
$2^{-M}<r$;
\item
\label{N-entropy-bound}
$\frac{1}{N} < \delta;$
\item
\label{it:speclen}
the gap function satisfies $L_r(N)<M$;
\item
\label{it:mark}
Lemma \ref{marker-lemma} holds with $\alpha=\delta^2/22$ and the value of $M$ as defined in \eqref{theM};
\end{enumerate}
\begin{enumerate}[(1)]
\item $\mu(S_{1,N})>1-\delta$, 
where
\[
S_{1,N}:=\{x\in X\colon \mu(B_0^N(x,\eta))>e^{-(h_\mu(T)+\Delta)N}\};
\]
\label{it:bigboys}
\item ${\pi_2^*(\xi)}(S_{2,N})>1-\delta$,
 where
\[
S_{2,N}:=\{y\in Y\colon {\pi_2^*(\xi)}(B_M^{N-10M}(y,2\eta))<
e^{-(h_{\pi_2^*(\xi)}(S)-\Delta)N}\};
\]
\label{it:smallgirls}
\item $\xi(S_{3,N})>1-\delta$,
 where
\[
S_{3,N}:=\{(x,y) \in X \times Y\colon \xi(B_M^{N-10M}(x,y),2\eta)<
e^{-(h_\xi(T \times S)-\Delta)N}\};
\]
\label{it:smallprod}
\item $\xi(S_{4,N})>1-\delta$,
 where
\[
S_{4,N}:=\{(x,y) \in X \times Y\colon \xi(B_0^N((x,y),\eta))>
e^{-(h_\xi(T \times S)+\Delta)N}\};
\]
\label{it:bigprod}
\item $\xi(S_{5,N})>1-\delta$,
 where
\[
S_{5,N}:=\big\{ (x,y) \colon
|\avg_M^{N-10M}f(x,y)-\xi(f)|<\tfrac{\epsilon}{12} \text{ for all }
f \in \mathrm{Lip}_1(X\times Y)\big\};
\]
\label{it:weakstar}
\item ${\pi_2^*(\xi)}(S_{6,N})>1-\delta$,
 where
\[
S_{6,N}:=
\{y \in Y\colon \avg_M^{N-10M}\mathbf 1_{\partial_r\mathcal Q_\ell}(y)<\delta\}.
\]
\label{it:offboundary}
\end{enumerate}
\end{lemma}
Note that in Lemma \ref{choose-N}, our  choices of constants give:
\begin{equation}
\label{choice-of-con}
120r, \ 80\delta, \  12 \eta \   < \ \e.
\end{equation}

\begin{proof}[Proof of Lemma \ref{choose-N}]
Conditions \eqref{easycond-b}, \eqref{easycond} and \eqref{N-entropy-bound} are trivial.
By Lemma \ref{marker-lemma}, for all sufficiently large $M$ (hence for sufficiently large $N$),
the conditions are satisfied establishing \eqref{it:mark}.     Condition \eqref{it:speclen} holds
for large $N$ as a consequence of the definition of almost weak specification.

Conditions \eqref{it:bigboys} and \eqref{it:bigprod}
follow for large $N$ from the upper estimate in Theorem \ref{thm:BK}.

For \eqref{it:smallgirls},
%let $\nu:=\pi_2^*(\xi)$.
let
\[
S_{2,N}':=\{y\in Y\colon {\pi_2^*(\xi)}(B_0^{N-11M}(y,2\eta))<
e^{-(h_{\pi_2^*(\xi)}(S)-\Delta)N}\},
\]
and note that ${\pi_2^*(\xi)}(S_{2,N}') = {\pi_2^*(\xi)}(S_{2,N})$.  Thus it suffices to show that
\begin{equation}
\label{transfer-M}
{\pi_2^*(\xi)}(S_{2,N}') > 1- \delta \text{ for sufficiently large } N.
\end{equation}

For ${\pi_2^*(\xi)}$-almost every $y$, we have
\begin{align*}
\liminf_{N\to\infty}&-(1/N)\log{\pi_2^*(\xi)}(B_0^{N-11M}(y,2\eta))
=(1-\delta)\underline h_{\pi_2^*(\xi)}^\text{BK}(2\eta)  \\
&>(1-\delta)(h_{\pi_2^*(\xi)}(S)-\delta) \ (\text{by} \  \eqref{BK-bg})\\
& >h_{\pi_2^*(\xi)}(S)-\delta(1+h_{\pi_2^*(\xi)}(S))\\
&>h_{\pi_2^*(\xi)}(S)-\Delta   \ (\text{by} \  \eqref{en-cap-delta}).
\end{align*}
Hence \eqref{transfer-M} follows.

A similar argument shows \eqref{it:smallprod} (using \eqref{BK-DD})  holds for large $N$.

That condition \eqref{it:offboundary}  holds for large $N$, follows from
the Birkhoff ergodic theorem and \eqref{boundarychosen}.
That condition \eqref{it:weakstar} holds for large $N$, would follow immediately
from the Birkhoff ergodic theorem, if $\mathrm{Lip}_1$ were a finite set.
Recall that by $\mathrm{Lip}_1$, we mean those functions with Lipschitz constant 1 taking values in $[0,1]$.
By the Arzel\`a-Ascoli theorem, $\mathrm{Lip}_1$
is totally bounded with respect to the uniform norm on continuous functions.
Hence there exists a finite collection $F \subset \mathrm{Lip}_1$
that is $\epsilon/{36}$-dense in $\mathrm{Lip}_1$ with respect to the uniform norm.
Now set
\[
S_{5',N}:=\{(x,y)\colon |\avg_M^{N-10M} f(x,y)- \xi(f)|<\epsilon/{36}\text{ for all }f \in F\}.
\]
An application of the triangle inequality shows that $S_{5,N}\supset S_{5',N}$.
Notice that since $M=cN$, $\avg_M^Nf=N/(N-M)\cdot (1/N)(N\avg_0^Nf-cN\avg_0^{cN}f)$,
so that for a fixed $f\in\mathrm{Lip}_1(X\times Y)$, we have  $\avg_M^Nf(x)\to \xi(f)$ $\xi$-a.e.
Since $F$ is finite, we have that $\xi(S_{5',N})\to 1$ as $N\to\infty$ by the Birkhoff ergodic theorem.
\end{proof}

Lemma \ref{choose-N} and the following variant of Hall's marriage theorem will be used to
define the injective map from blocks of $X$ to orbit segments of $Y$.   Given a relation
$R \subset \boys \times \girls$, we let $R(b, \cdot) := \ns{g  \in \girls:  (b,g) \in R}$
and $R(\cdot, g):= \ns{b \in \boys: (b,g) \in R}$.
\begin{theorem}[Hall's marriage theorem \cite{Hall01011935}]
\label{hall}
Let $\boys$ and $\girls$ be finite sets.    Let $R \subset \boys \times \girls$ be a relation
with the property that there exists  $K>0$ such that for all $b \in \boys$ we have
$|R(b, \cdot)| \geq K$, and for all $g \in \girls$ we have $|R(\cdot, g)| \leq K$.
Then there exists an injection $\phi: \boys \to \girls$ such that as a relation
$\phi \subset R \subset \boys \times \girls$.
\end{theorem}
A proof of Theorem \ref{hall} can be found in  Downarowicz's book \cite[Appendix A]{MR2809170}.
In Theorem \ref{hall}, the sets $\boys$ and $\girls$ are often referred to as boys and girls,
respectively; we call the integer $K$ a {\df marriage bound} for $R$, and the map $\phi$ a {\df dictionary}.

We say that a subset   $A \subset Y$ is    {\df $\boldsymbol{(\metric_m^n,\eta)}$-{separated}}
if $\metric_m^n(a,a')\ge \eta$
for each distinct pair $a,a'$ in $A$; and is  said to be a {\df $\boldsymbol{(\metric_m^n,\eta)}$-{spanning}}
subset of $B \subset Z$ if $A\subset B$ and for each $b\in B$, there exists an $a\in A$ with $\metric_m^n(a,b)\le \eta$.
Notice that a maximal $(\metric_m^n,\eta)$-separated subset of $B$ is necessarily $(\metric_m^n,\eta)$-spanning.

Given a finite word, $B=b_0\ldots b_{N-1}$, with symbols in the alphabet of $X$,
we write $[B]$ for the cylinder set
$\{x\in X\colon x_0=b_0,\ldots,x_{N-1}=b_{N-1}\}$.

\begin{corollary}[Corollary to Lemma \ref{choose-N}]
\label{dict}
Fix $N$  such that conditions of Lemma \ref{choose-N} hold, with a point
$y_\mathrm{mark}\in Y$ satisfying \eqref{it:mark}.
Let $H_1$ and $H_2$ be defined as in Lemma \ref{marker-lemma}.  Let
\[
A:=\{y\colon T^jy\not\in H_1\cup H_2 \text{ for } 0\le j<N\}.
\]
Let $\girls$ be a maximal $(\metric_M^{N-10M},\eta)$-separated subset of $A\cap S_{2,N} \cap S_{6, N}$.
Let $U \subset X \times Y$ be given by
\[
U:=(X\times (S_{2,N}\cap S_{6,N}\cap A))
\cap (S_{1,N}\times Y)\cap S_{3,N}\cap S_{4,N}\cap S_{5,N}.
\]
Let $\boys$ be the set of elements $B \in \mathcal P^N$ such that $\mu([B])\ge e^{-N(h_\mu(T)+\Delta)}$ and
$\xi(([B]\times Y)\cap U)\ge \frac12 \mu([B])$.

Define
\begin{align*}
    R &:=\left\{(B,y)\in\boys \times \girls  \colon
    \exists(u,v)\in U\text{ with }u\in[B],\
    \metric_M^{N-10M}(y,v)<\eta\right\}.\\
    K &:=\tfrac{1}{2} e^{N(h_\xi(T\times S)
 -h_\mu(T) -2\Delta)}.
\end{align*}
Then
\begin{enumerate}[(A)]
\item
\label{largeU}
$\xi(U)>1-7\delta$;
\item
\label{G}
$|\girls|>
\frac12e^{N(h_{{\pi_2^*(\xi)}}(S)-\Delta)}$;
\item
\label{largeBB}
$|\boys| \leq e^{N(h_\mu(T)+\Delta)}$ and $\mu\left(\bigcup_{B\in\boys}[B]\right)>1-15\delta;$
\item
\label{KMB}
$K$ is a marriage bound for $R$;
\item
\label{KHall}
there exists an injection $\phi: \boys \to \girls $ such that $\phi \subset R$;
\item
\label{eq:wstarest}
if $B\in \boys$, then for all $x \in [B]$, $y \in B_{M}^{N-10M}(\phi(B),r)$, and for all
$f \in \mathrm{Lip}_1(X \times Y)$, we have
\begin{equation*}
\left|\avg_M^{N-10M}f(x,y)-\xi(f) \right|< \tfrac\epsilon{12}+\eta+r.
%\label{eq:wstarest}
\end{equation*}
\end{enumerate}
\end{corollary}

\begin{proof} 
\begin{enumerate}[(A)]
\item 
Since $\xi$ is a joining, from Lemma \ref{choose-N} \eqref{it:smallgirls}, \eqref{it:offboundary},
and  \eqref{it:mark},  we see that $\xi(X\times (S_{2,N}\cap S_{6,N}\cap A))>1-3\delta$ and
from Lemma \ref{choose-N} \eqref{it:bigboys} $\xi(S_{1,N}\times Y)>1-\delta$ so that Condition \eqref{largeU} holds.
\item
By the definition of $\girls$ we have,
\[
\bigcup_{y\in\girls}
B_M^{N-10M}(y,\eta)\supset A \cap S_{2,N}\cap S_{6,N}.
\]
By Lemma \ref{choose-N}
\eqref{it:smallgirls} we have for each $y\in\girls$, ${\pi_2^*(\xi)}(B_M^{N-10M}(y,\eta))
% \leq \nu(B_M^{N-10M}(y,2\eta))
<e^{-N(h_{\pi_2^*(\xi)}(S)-\Delta)}$.
  We deduce $|\girls|>
\frac12e^{N(h_{\pi_2^*(\xi)}(Y)-\Delta)}$.
\item
The first claim follows from the definition of $\boys$.   For the second, let
$\boys_1=\{B\in\mathcal P^N\colon \mu([B])>e^{-(h_\mu(T)+\Delta)N}\}$. For
$x\in[B]$, we have $[B]=B(x,N,\frac12)\supseteq B(x,N,\eta)$. By Lemma \ref{choose-N} \eqref{it:bigboys},
we have $\mu(\bigcup_{B\in\boys_1}[B])>1-\delta$. We then have
\begin{align*}
7\delta&>\xi(U^c)\ge \xi\Big(\bigcup_{B\in \boys_1\setminus\boys}([B]\times Y)\cap U^c\Big)
\ge \sum_{B\in \boys_1\setminus \boys}\textstyle{\frac12}\mu([B]).
\end{align*}

This yields $\mu\Big(\bigcup_{B\in \boys}[B]\Big)>1-15\delta$ as required.

\item
Let $B\in\boys$ and let $R(B,\cdot)=\{y_1,\ldots,y_n\}$.   We will show that $n \geq K$.
Since $(B,y_j)\in R$, it follows there exist $(u_j,v_j)\in U$ such that $u_j\in[B]$ and
$\metric_{M}^{N-10M}(v_j,y_j)<\eta$.
Let $S_j=B_M^{N-10M}((u_j,v_j),2\eta)$. Since $(u_j,v_j)\in U$, by Lemma
\ref{choose-N} \eqref{it:smallprod},
$\xi(S_j)<e^{-(h_\xi(T \times S)-\Delta)N}$.
We claim that the $S_j$ cover $([B]\times Y)\cap U$.
To see this, let $(u,v)\in U$ satisfy $u\in [B]$.
Since $(u,v)\in U$, we have $v\in A\cap S_{2,N}\cap S_{6,N}$.
Hence, by the definition of $\girls$, there exists $y \in\girls$ with
$\metric_M^{N-10M}(y,v)<\eta$, so that $(B, y) \in R$; hence  $y = y_j$ for some $j$.
It follows that $\metric_M^{N-10M}(v,v_j)<2\eta$. Since $u,u_j\in[B]$, we have
$\metric_M^{N-10M}(u,u_j)<2^{-M}<\eta$ so that $\metric_M^{N-10M}((u,v),(u_j,v_j))<2\eta$
and $(u,v)\in S_j$ as required.
Since $B\in\boys$, we have $\xi(([B]\times Y)\cap U)\ge \tfrac12\mu([B])\ge \frac12e^{-(h_\mu(T)+\Delta)N}$.
Since $([B]\times Y)\cap U$ is covered by the $S_j$'s, we see that
$n\ge \frac12e^{N(h_\xi(T\times S)-h_\mu(T)-2\Delta)}=K$ as required.

For the other half of the argument, let $y\in\girls$ and let $R(\cdot,y)=\{B_1,\ldots,B_m\}$.
Pick witnesses $(u_i,v_i)\in U$ so that $u_i\in[B_i]$ and $d_M^{N-10M}(v_i,y)<\eta$.
Let $D_i=B_0^N((u_i,v_i),\eta)$. These sets are all contained in $X\times
B_M^{N-10M}(y,2\eta)$, which, since $y\in\girls$, has measure at most
$e^{-(h_{\pi_2^*(\xi)}(S)-\Delta)N}$ by Lemma \ref{choose-N} \eqref{it:smallgirls}.
Since $(u_i,v_i)\in U$, by Lemma \ref{choose-N} \eqref{it:bigprod},  each $D_i$
has measure at least $e^{-(h_\xi(T\times S)+\Delta)N}$.
Finally, if $i\ne i'$, then $\metric_0^N(u_i,u_{i'})=1$ so that the $D_i$ are disjoint.
In particular, we deduce from  \eqref{def-cap-delta} and Lemma \ref{choose-N}
\eqref{easycond-b} that $m\le e^{N(h_\xi(T \times S)-h_{{\pi_2^*(\xi)}}(S)+2\Delta)}<K$.
\item
Property \eqref{KHall} follows immediately from Property \eqref{KMB} and Theorem \ref{hall}.
\item
Let $B \in \boys$ and $x \in [B]$.   Let  $f \in \mathrm{Lip}_1$.
Since $(B, \phi(B)) \in R$, there exists    $(x_0,y_0)\in U$  such that
$x_0 \in [B]$ and $y_0 \in B_M^{N-10M}(\phi(B), \eta)$.    Thus from the definition
of $U$ and Lemma \ref{choose-N} \eqref{it:weakstar}, we have
\[
|\avg_M^{N-10M}f(x_0,y_0)-\xi(f)|<\epsilon/{12}.
\]
If $x\in[B]$ and $y$ is any point
in $B_M^{N-10M}(\phi(B),r)$, then by condition \eqref{easycond} of Lemma \ref{choose-N} and \eqref{productmetric}, we have $\metric((T^ix,S^iy),(T^ix_0,S^iy_0))<r+\eta$
for each $M\le i<N-10M$,
so that by the Lipschitz property,
\[
\left|\avg_M^{N-10M}f(x,y)-\xi(f) \right|< \tfrac\epsilon{12}+\eta+r,
\]
as required.
\qedhere
\end{enumerate}
\end{proof}

\section{Proof of Proposition \ref{enl-open} \eqref{dense}}

In this section, we use Corollary \ref{decipher}, Lemma \ref{choose-N} and Corollary \ref{dict} to build a
joining. Using Lemma \ref{up}, we then prove the following restatement of Proposition
\ref{enl-open} \eqref{dense}, where we assume a strict entropy gap.

\begin{lemma}
\label{build}
Consider the setup of Lemma \ref{choose-N}.   Let $(X, \mu, T)$ be a non-trivial
ergodic subshift with invariant measure $\mu$ and natural generating partition $\mathcal P$.
Let $S$ be a self-homeomorphism with almost weak specification on a compact metric space $Y$
that satisfies the small boundary condition witnessed by a sequence of refining partitions
$(\mathcal Q_{\ell})$, each having zero measure boundary, where
$\mathrm{diam}(\mathcal Q_{\ell}) < 1/\ell$.
   Let $\xi$ be a $\mu$-joining with
%$\nu:=\pi_2^*\xi $ and
$h_{\pi_2^*(\xi)}(S) > h_{\mu}(T)$.    Let $\e >0$.   There exist $\ell$ and
 a $\mu$-joining, $\tilde \xi$,  satisfying the following  properties:
\begin{itemize}
\item {weak$^*$-closeness}: $\metricstar(\xi, \tilde \xi) < \e$;
\item {approximate  embedding  properties}: $\mathcal
Q_\ell\stackrel{\e}{\subset} \bigvee_{i\in \Z} \mathcal T^{-i} P \mod \tilde \xi;$
and $\mathcal P\stackrel{\e}{\subset}  \bigvee_{i\in \Z}{S^{-i}\mathcal  Q_\ell} \mod \tilde \xi$;
\item {entropy preservation}:  $h_{\pi_2^*(\tilde\xi)}(S) \ge h_\mu(T)$.
\end{itemize}
\end{lemma}

\begin{proof}[Proof of Proposition \ref{enl-open} \eqref{dense}]
Let $\xi \in \mathcal M_0$.   Let $n >0$.   Without loss of generality,
 let $0 < \e< 1/n$.   We need to find a $\tilde \xi \in E_{\mu} ^n$ such that $\metricstar(\xi, \tilde \xi) < \e$.
By Lemma \ref{up}, we may assume without loss of generality that $h_{\pi_2^*(\xi)}(S) > h_\mu(T)$.
So the result follows from Lemma \ref{build}.
\end{proof}

We need one more tool before we can define the joining in the proof of Lemma \ref{build}.
We will make use of the following variation of the Rokhlin tower theorem.

\begin{lemma}[Rokhlin tower theorem:  independent base version]
\label{rok}
Let $(\X, \mathcal F, \mu, T)$ be a non-periodic measure-preserving system.
Let $N$ be a positive integer and $\delta >0$.   For any finite measurable partition $\mathcal W$,
there exists $F \in \mathcal F$ (the \emph{base}) with the following properties.
\begin{enumerate}
\item
The sets $F, TF, \ldots, T^{N-1}F$ are pairwise disjoint.
\item
The complement of their union $E_0:=\X \setminus \bigcup_{i=0} ^{N-1} T^i F$
(the \emph{error set})  has measure exactly $\delta$.
\item
The $\sigma$-algebras generated by $F$ and $\mathcal W$ are independent, so that $\mu(F \cap W) =
\mu(F)\mu(W)$ for all $W \in \mathcal W$.
\end{enumerate}
\end{lemma}

For a proof, see the book of McCutcheon and Kalikow \cite[Theorem 184]{Kalikow}.
After defining the joining in the proof of Lemma \ref{build}, we motivate why it satisfies
the desired properties before we proceed with technical calculations.

\begin{proof}[Proof of Lemma \ref{build}: definition of $\tilde \xi$]
 Choose  $\Delta, \delta, \eta, \ell, r, M$, and $N$ %sufficiently large
 as in  Lemma \ref{choose-N}.
By Lemma \ref{rok}, let $F$ be the base of a Rokhlin tower in $X$ with height $N$ and error set
$E_0$ of measure
$\delta$ with the property that $F$ is independent of $\bigvee_{i=0}^{N-1}T^{-i}\mathcal P$,  where $\mathcal P$ is the natural
generating partition for $X$.
Given a point $x$ that we assume to be generic for $\mu$, we let
$n_0(x):=\sup\{k\le 0\colon T^kx\in F\cup E_0\}$. We then let $(n_j(x))_{j\in\Z}$
be the enumeration of $\{k\in\Z\colon T^kx\in F \cup E_0\}$ satisfying
\[
\ldots<n_{-2}(x)<n_{-1}(x)<n_0(x)<n_1(x)<\ldots.
\]
The {\df $\boldsymbol{j}$th  block of $\boldsymbol{x}$} is then the block
\[
B_j(x):=x_{n_j(x)}x_{n_j(x)+1}\ldots x_{n_{j+1}(x)-1}.
\]
The blocks of $x$ are of length $N$ (for those
$j$ such that $T^{n_j(x)}(x)\in F$) and $1$ for those points on the orbit that land in the error set $E_0$.
%Let $M=\delta N/11$.

Let $y_{\mathrm{mark}}$, $\boys$, $\girls$, and   $\phi$ be  given by
Corollary \ref{dict}.    We introduce special symbols
$\follow, \nullcon \not \in Y$ (with $\follow$ standing for `dictionary'
and $\nullcon$ standing for `vacuous'). Let  $\girls_\follow=\girls \cup\{\follow\}$
and $Y_\nullcon=Y\cup\{\nullcon\}$.
Equip $\girls_\follow^\Z$ with the usual shift map $\sigma$, and  the Bernoulli measure $\zeta$
giving each coordinate mass $\epsilon/(2|\girls|)$ to each member of $\girls$ and $1-\epsilon/2$ to $\follow$.
If $n_{j+1}(x)-n_j(x)=N$, let $m_j(x):= n_j(x) + N - 9M$ and set
\begin{align*}
R^\text{info}_j(x)&:=[n_j(x)+M,n_j(x)+N-10M)\text{; and}\\
R^\text{mark}_j(x)&:=[m_j(x),m_j(x)+8M);
\end{align*}
otherwise, let $R^\text{info}_j(x)=R^\mathrm{mark}_j(x)=\emptyset$.
%Let $n_j'(x)=n_j(x)+N-9M$.
We now define a map $\psi :X\times \girls_\follow^\Z \to Y_\nullcon^\Z$ by
\begin{equation*}
\psi(x,z)_n:=
\begin{cases}
S^{n-n_j(x)}z_{n_j(x)}&\text{$n\in R^\text{info}_j(x)$, $B_j(x)\in\boys$, $z_{n_j(x)}\in \girls$;}\\
S^{n-n_j(x)}\phi(B_j(x))&\text{$n\in R^\text{info}_j(x)$, $B_j(x)\in\boys$, $z_{n_j(x)}=\follow$;}\\
S^{n-m_j(x)}y_\mathrm{mark}&\text{$n\in R^\mathrm{mark}_j(x)$, $B_j(x)\in\boys$, $z_{n_j(x)}=\follow$;}\\
\nullcon &\text{otherwise.}
\end{cases}
\end{equation*}

We wish to $r$-shadow elements of the sequence $\psi(x,z)$ that are in $Y$.    By condition \eqref{it:speclen} of Lemma \ref{choose-N},
the orbit segments that we are attempting to shadow are of length less
than $N$ and are separated by at least $M>L_r(N)$.  Note also that $\mu \times \zeta$ is ergodic, since $\zeta$ is mixing.       Thus by Proposition \ref{interpolation}, on the  product space $\Omega$ with an  invertible measure-preserving transformation $\tau$ defined via:
\[
\Omega:=X
\times \girls_\follow^\Z \times Y \ \text{and} \ \tau:= T \times \sigma \times S,
\]
there exists an ergodic invariant
measure $\iota$ on $\Omega$
such that for $\iota$-almost all points $\omega = (x,z,y)$, we have
\begin{equation}
\label{spec-con}
\metric (S^n y,
\psi (x, z)_n)
\leq r \text{ whenever }\psi ( x, y)_n \ne \nullcon, \text{ for all } n\in \Z.
\end{equation}
Define $\tilde \xi$ to be the ergodic measure obtained by projecting $\iota$ onto the first and
last coordinates of the tuple.  Since Proposition \ref{interpolation} gives that the projection of $\iota$ on $X \times \girls_\follow^\Z$ is $\mu \times \zeta$, we have that
  $\tilde\xi$ is a $\mu$-joining.
\end{proof}

It remains to verify that $\tilde \xi$ satisfies the required properties.
Using Corollary \ref{dict}   \eqref{eq:wstarest} on the blocks of
length $N - 11M$ where we are applying the dictionary $\phi$ on $\boys$, we will obtain
the weak$^*$-closeness by Corollary \ref{dict} \eqref{largeBB} and \eqref{eq:wstarest}.

We consider a $\tilde\xi$-typical point $(x,y)$.
If we are given $y$, the fact that we are placing markers in between the coded boys
will enable us to decide where the information is encoded, and hence
to recover most of $x$, giving us the latter half of the approximate embedding  property.
Suppose we know $x$ and want to guess which element of 
$\mathcal Q_\ell$ the point $y$ belongs to.
We have no chance if $B_0(x)$ does not belong
to $\boys$; if it does, and the dictionary $\phi$ is applied, instead of using a random element of
$\girls$, then we know that $(S^ky)$ shadows $(S^{k-n_0(x)}\phi(B_j(x)))$ in $R^\text{info}_0(x)$.      Unless
$S^{n_0(x)}\phi(B_j(x))$ lies close to the boundary of $\mathcal Q_\ell$, we can deduce
which element $y$ lies in, giving us the
other approximate embedding property.

Finally, the entropy of $S$ with respect to $\pi_2 ^* (\tilde \xi)$ has a lower bound that is the sum of two contributions:
one term is the entropy obtained by encoding most of the blocks of $X$,
and another is given by sometimes using random elements of $\girls$.
We do not encode blocks that are not members of $\boys$ and
because of the error set in the Rokhlin's lemma the contribution from encoding blocks of $X$
is strictly less than $h_{\mu}(T)$; however, this entropy loss is at most
\[
\left(\mu\Big(\bigcup_{B \not\in \boys} [B]\Big) + \mu(E_0) \right)  \cdot\log |\mathcal P|.
\]
For an $\epsilon/2$ proportion of blocks that do belong to $\boys$, we do not apply the
dictionary $\phi$, but instead use
a random element of $\girls$; we will see that this results in a net gain of entropy that is large
enough to cover the losses  incurred by encoding only members of $\boys$ and by not
being able to encode anything on the set $E_0$.

\begin{proof}[Proof of Lemma \ref{build}: weak$^*$-closeness]

Let $W \in \bigvee_{i=0} ^{N-1} T^{-i} \mathcal P$.
Recall that by construction, we have
$\mu(W \cap F) = \mu(W)\mu(F)$.
Since $T$ preserves the measure $\mu$ and for all $ 0\leq k \leq N-1$
\[
\ns{x \in X: T^{-k}x \in F} = \ns{ x \in X: x \not \in E_0, n_0(x) =-k},
\]  
we have
\begin{equation}
\label{indep}
 \mu\ns{ x \in X: T^{n_0(x)}x \in W, x \not \in E_0, n_0(x) = -k}= \mu(W)\mu(F)
 \end{equation}
for all $ 0 \leq k \leq N-1$ and hence
\begin{equation}
\label{indep:two}
\mu\ns{x \in X : T^{n_0(x)}x \in W:  x \not \in E_0} = \mu(W)(1-\delta).
\end{equation}
Define the `bad set' by
\begin{align*}
\badset_1 &:= \{(x,z,y)\in\Omega\colon x\in E_0\}\\
\badset_2 &:= \{ (x,z,y)\in\Omega\colon x\not \in E_0,  x_{n_0(x)} \cdots x_{n_0(x) + N -1} \not \in \boys\} ;\\
\badset_3 &:=\{(x,z,y)\in\Omega\colon x\not\in E_0, z_{n_0(x)}\ne \follow\};\\
\badset_4 &:= \{(x,z,y)\in\Omega\colon x\not \in E_0, 0  \not\in R_0^\text{info}(x)\};\\
\badset&:=\badset_1\cup \badset_2\cup \badset_3\cup\badset_4.
\end{align*}

We have that $\iota(\badset_1) = \mu(E_0) = \delta$.    By \eqref{indep:two} and
Corollary \ref{dict} \eqref{largeBB}, we get
\[
\iota(\badset_2)=(1-\delta)\mu\bigg(\bigcup_{B\in\mathcal P^N\setminus\boys}
[B] \bigg)\le (1-\delta)15\delta.
\]    
By independence of $\zeta$ and $\mu$, we have
$\iota(\badset_3) = (1-\delta)(\e/2)$.    Using the fact that the length of the interval $R_0^{\text{info}}(x)$
is $N-11M$ when $x\not \in E_0$,  by our choice of $M= \lfloor \delta N / 11 \rfloor$ in
 \eqref{theM}, we have $\iota(\badset_4) \leq \delta$.    Thus by \eqref{choice-of-con}, we have
\begin{equation}
\label{badseteq}
 \iota(\badset) \le 17\delta + \e/2 <3\epsilon/4.
\end{equation}

Let $f\in\mathrm{Lip}_1(X\times Y)$. For $0\le k<N$, let
$A_k=\{(x,z,y)\colon x\in \badset_1^c\cap \badset_2^c\cap \badset_3^c,\, n_0(x)=-k\}$.   Notice that
\begin{align*}
\badset^c &= \badset_1^c \cap \badset_2^c \cap \badset_3^c \cap \badset_4^c \\
& = \badset_1^c \cap \badset_2^c \cap \badset_3^c \ \cap \ {\bigcup_{k=M}^{N-10M-1}
\ns{(x,z,y) \in \Omega: n_0(x)=-k}} \\
& = {\bigcup_{k=M}^{N-10M-1} A_k}.
\end{align*}
So we have
\begin{align*}
&\int f(x,y)\,d\tilde\xi(x,y)=\int f(x,y)\,d\iota(x,z,y)\\
&=\int_{\badset}f(x,y)\,d\iota(x,z,y)+\sum_{k=M}^{N-10M-1}\int_{A_k}f(x,y)\,d\iota(x,z,y)\\
&=\int_\badset f(x,y)\,d\iota(x,z,y)+(N-11M)\int_{A_0}
\avg_M^{N-10M}f(x,y)\,d\iota(x,z,y).
\end{align*}

By Corollary \ref{dict} \eqref{eq:wstarest}, for $(x,z,y)\in A_0$, we have
$|\avg_M^{N-10M}f(x,y)-\xi(f)|<\e/12+\eta+r$. We also have $(N-11M)\mu(A_0)
=\mu(\badset^c)$. Hence by \eqref{badseteq}  and \eqref{choice-of-con}, we have
\begin{equation*}
|\tilde\xi(f)-\xi(f)| < \mu(\badset)+\epsilon/{12}+\eta+r<\epsilon.
\qedhere
\end{equation*}
\end{proof}

\begin{proof}[Proof of Lemma \ref{build}: approximate embedding properties]
We first show that $\mathcal Q_\ell\stackrel{\epsilon}{\subset}{ \bigvee_{i \in \Z}{T^{i}\mathcal P}}  \mod \tilde \xi$.
Enumerate $\mathcal Q_\ell$ as $\{D_1,\ldots,D_n\}$ and
regard $\mathcal Q_\ell$ as a map from $Y$ to $\{1,\ldots,n\}$
where $\mathcal Q_\ell(y)=j$ if $y\in D_j$.
Define
\begin{equation*}
\chi(x):=\begin{cases}
	S_0^{-n_0(x)}\phi(B_0(x))&\text{if it's defined;}\\
	y_\mathrm{mark}&\text{otherwise.}
	\end{cases}
\end{equation*}
We define  $\mathcal Q'$ from $X$ to $\{1,\ldots,n\}$ by
\begin{equation*}
\mathcal Q'(x):=\mathcal Q_\ell(\chi(x)).
\end{equation*}
Since $\mathcal Q'$ is $\sigma(\bigvee_{i \in \Z} T^{i}\mathcal P)$-measurable, it suffices to show that
\[\iota(\{(x,z,y)\colon \mathcal Q_\ell(y)\ne \mathcal Q'(x)\})<\epsilon.\]

Let
\[\badset_5:= \{(x,z,y)\not\in\badset\colon \chi(x)\in\partial_r\mathcal Q_\ell\}.\]
Notice that by definition of $\iota$, we have
\[\{(x,z,y)\colon \mathcal Q_\ell(y)\ne \mathcal Q'(x)\}
\subset \badset\cup \badset_5\mod\iota,\]
By the definition of the set $\girls$ and $\phi$, we have that   $\phi(B)\in S_{6,N}$ for
all $B\in\boys$.   Thus by Lemma \ref{choose-N} \eqref{it:offboundary} we have
$\avg_{M}^{N-10M}\mathbf 1_{\badset_5}(x,z,y)<\delta$ on $A_0$.
We compute as in the weak$^*$-closeness section.
\begin{align*}
\iota(\badset_5)&=\sum_{k=M}^{N-10M-1}\iota(\badset_5\cap A_k)\\
&=(N-11M)\int_{A_0}\avg_{M}^{N-10M}\mathbf 1_{\badset_5}(x,z,y)\,d\iota(x,z,y)
<\delta.
\end{align*}
  By \eqref{badseteq} and \eqref{choice-of-con}, we have
\begin{equation*}
\tilde\xi\{(x,y)\colon \mathcal Q_\ell(y)\ne\mathcal Q'(x)\}\le \iota(\badset\cup \badset_5)\le \tfrac{3}{4}\e+\delta<\epsilon.
\end{equation*}

Next, we show the approximate embedding  in the opposite direction.
Let  $\q:=\bigvee_{i\in \Z}{S^{-i}\mathcal  Q_\ell}$.   (Thus $\q(y) \in \q$ is the part which contains $y$.)   We need to show that  $\mathcal P\stackrel{\e}{\subset} \q \mod \tilde \xi$.    The proof relies on the  markers
 and the invertibility of $\phi$.  It suffices to define a  function $\tilde{\mathcal P}:Y \to X$ and a set $\badset_6$ such that
\begin{enumerate}
\item
\label{meas-BBB}
$\iota(\badset \cup \badset_6) < \e$,
\item
\label{first-y}
$\{(x,z,y)\colon \mathcal P(x)\ne\tilde{\mathcal P}(y)\} \subset \badset \cup \badset_6$, and
\item
\label{equivclass}
for $\iota$-a.e.\ $(x,z,y)\in\Omega\setminus(\badset\cup \badset_6)$,  if $y' \in \q(y)$, then
$\tilde{\mathcal P}(y) = \tilde{\mathcal P}(y').$
\end{enumerate}

We let 
\[
\badset_6:=\{(x,z,y)\in\Omega\colon n_0(x)<-N+9M\}.
\] 
We see
$\iota(\badset_6)<9M/N<\delta$.  From \eqref{badseteq} and \eqref{choice-of-con},  we see that  property \eqref{meas-BBB} is satisfied.    Recall that  $0<r,\mathrm{diam}(\mathcal{Q}_{\ell}) < \eta/10$.  Thus for $\iota$-a.e.\ $(x,z,y)\in \Omega$, we have that if $y' \in \q(y)$, then $\metric(S^iy',S^iy) < \eta/10$ for all $i \in \Z$.
    Define
\[
\tilde n_0(y):=\min\big(\min\{k\ge 0\colon S^{k}(y)\in B_0^{8M-1}(y_\mathrm{mark},3\eta-r)\}-(N-9M),0\big).
\]
Equip $\boys$ with an arbitrary total order and define a map $b\colon Y\to\boys$ by
\begin{equation}
\label{argmin}
b(y):=\argmin_{B\in\boys}\metric_M^{N-10M}(y,\phi(B)),
\end{equation}
breaking ties lexicographically if necessary. Finally, set
\[
\tilde{\mathcal P}(y):=b(S^{\tilde n_0(y)}(y))_{-\tilde n_0(y)}.
\]

For $\iota$-a.e.\ $(x,z,y)\in\Omega\setminus(\badset\cup \badset_6)$, we have
$\tilde n_0(y)=n_0(x)$ by the definition of $\iota$, $\girls$ (see the beginning of Corollary \ref{dict}),  and Corollary \ref{decipher}; furthermore, we have that if $y' \in \q(y)$, then $\tilde n_0(y)=\tilde n_0(y')$.
By  \eqref{spec-con} and choice of parameters in Lemma \ref{choose-N}, for  $\iota$-a.e.\ $(x,z,y)\in\Omega\setminus(\badset\cup \badset_6) $, we have
$\metric_M^{N-10M}\big(y',\phi(B_0(x))\big) < r + \eta/10 < \eta/5$, for all $y' \in \q(y)$,  and since $\girls$  is a $(\metric_M^{N-10M},\eta)$-separated set, $B_0(x)$ realizes \eqref{argmin} and $b(y) = b(y')$ for all $y' \in \q(y)$.  Thus property \eqref{equivclass} is satisfied.
Since  $\phi$ is one-to-one, we have for  $\iota$-a.e.\ $(x,z,y)\in\Omega\setminus(\badset\cup \badset_6)$ that   $b(y)=B_0(x)$.  Thus property \eqref{first-y} is satisfied.
%
%Hence we see that
%\begin{equation*}
%\tilde\xi(\{(x,y)\colon \mathcal P(x)\ne\tilde{\mathcal P}(y)\})\le %\iota(\badset\cup \badset_6)\le 19\delta<\epsilon.
%\qedhere
%\end{equation*}
%
%
\end{proof}

\begin{proof}[Proof of Lemma \ref{build}: entropy preservation]
We define a number of partitions of $\Omega$ that we shall need in order to do
calculations.       We will express a typical point of $\omega \in \Omega$ as $\omega = (x,z,y)$.    Define
\[
C_1:=\{\omega\in\Omega\colon x \not \in E_0, B_0(x)\in \boys
\text{ and }  z_{n_0(x)} = \follow\},
\]
and
\[
C_2:=\{\omega\in\Omega\colon x \not \in E_0, B_0(x)\in\boys \text{ and }
z_{n_0(x)} \not =\follow\}.
\]
Let $C_0:=\Omega\setminus(C_1\cup C_2)$, and $\mathcal C:=\ns{C_0,C_1,C_2}$.   Set
\[
\mathcal R := \sigma\left( \bigvee_{i \in \Z} \tau^{-i} \mathcal C\right).
\]

 By regarding  the partition $\mathcal Q_\ell$ of $Y$
as a partition of $\Omega$, we let
\[
\mathcal Q_{\ell}^{1,2}:=\big\{Q \cap  C_0^c  \colon Q\in\mathcal Q_\ell \big\}
\cup \big\{ C_0 \big\}.
\]
Also, for each $j \in \ns{1,2}$, let
\[
\mathcal Q_{\ell}^{j}:=\big\{Q \cap  C_j  \colon Q\in\mathcal Q_\ell \big\}
\cup \big\{ \Omega \setminus C_j \big\}.
\]

Recall that $\psi$ may take the value $\nullcon$.    Let
%We add the extra part $\ns{\nullcon}$ to ${\mathcal Q}_{\ell}$ by taking  $[{\mathcal Q}_{\ell}; \nullcon]: = {\mathcal Q}_{\ell} \cup\ns{\ns{\nullcon}}$.
\[
\tilde{\mathcal Q}_{\ell}^{1,2}:=\Big\{\{\omega \in C_1 \cup C_2 \colon
\psi(x,z)_0 \in Q\} \colon
Q\in  {\mathcal Q}_{\ell} \cup\ns{\ns{\nullcon}} \Big\}\cup \big\{  C_0 \} .
\]
For each $j \in \ns{1,2}$, let
\[
\tilde{\mathcal Q}_{\ell}^{j}:=\Big\{ \{\omega \in C_j \colon
\psi(x,z)_0 \in Q\}
\colon Q\in  {\mathcal Q}_{\ell} \cup\ns{\ns{\nullcon}}  \Big\}
\cup \big\{  \Omega \setminus C_j \}.
\]

Let $\bigvee_{i=M}^{N-10M-1}S^{-i}\mathcal Q_\ell=\{A_1,\ldots,A_L\}$.
Let ${\mathcal Q}_{\ell}^{\Bl 1,2}$ be the partition given  by
$\{ A_0^{\Bl 1,2},A_1^{\Bl 1,2},\ldots, A_L^{\Bl  1,2}\}$,
where
\[
A_i^{\Bl 1,2}:=\{\omega\in C_1\cup C_2\colon n_0(x)=0\text{ and }
 y \in A_i\}
\]
for $1\le i\le L$ and
$ A_0^{\Bl 1,2}$ is the complementary set.   (Here the `BL' stands for
`block.')

 Note that if $(x,z,y) \in C_1 \cup C_2$ and $n_0(x) = 0$, then  $\psi(x,z)_k \not = \nullcon$ for all $M \leq k \leq N-10M-1$ and thus $\psi(x,z)_0 \in A_i$ for some $i$.
Let $\tilde{\mathcal Q}_{\ell}^{\Bl   1,2}$  be the partition given by
\[
\{\tilde A_0^{\Bl     1,2},\tilde A_1^{\Bl  1,2},\ldots,\tilde
A_{L}^{\Bl   1,2}\},
\]
where
\[
\tilde A_i^{\Bl 1,2}:=\{\omega\in C_1\cup C_2\colon n_0(x)=0\text{ and }
\psi(x,z)_0\in A_i\}
\]
for $1\le i\le L$ and
$\tilde A_0^{\Bl  1,2}$ is the complementary set.
Similarly for $j\in\{1,2\}$, let
$\tilde{\mathcal Q}^{\Bl   j}$ be the partition
$\{\tilde A_0^{\Bl  j},\tilde A_1^{\Bl  j},\ldots,\tilde A_{L}^{\Bl   j}\}$ where
\[
\tilde A_i^{\Bl    j}:=
\{\omega\in C_j\colon n_0(x)=0\text{ and }
\psi(x,z)_0 \in  A_i\}
\]
and $\tilde A^{\Bl   j}_0$ is the complementary set.

Finally let $\mathcal P^{\Bl   1}$ be the partition with elements
\[
\ns{\omega \in C_1:  x \in P \text{ and } n_0(x) = 0},
\]
where $P \in \bigvee_{i=0}^{N-1}T^{-i}\mathcal P$, together with the complement
of the union of this collection; $\mathcal P^{\Bl   2}$ be the partition with elements
\[
\ns{\omega \in C_2:  x \in P \text{ and } n_0(x) = 0},
\]
where $P \in \bigvee_{i=0}^{N-1}T^{-i}\mathcal P$, together with the
complement; and $\mathcal P^0$ be the partition with elements $P
\cap\{\omega\colon B_0(x)\not\in\boys \text{ or } x \in E_0\} $, where
$P \in \bigvee_{i=0}^{N-1}T^{-i}\mathcal P$, again along with the
complement.

Let us pause to explain the above notations.  The partition $\mathcal
C$ tells you whether you are attempting to shadow an element of
$\girls$ determined by the dictionary $\phi$ ($C_1$), a random element
of $\girls$ ($C_2$) or if there is no constraint ($C_0$).  For the
$\mathcal Q_{\ell}$ partitions, the superscript $1$ indicates that you
are looking at those times when you are shadowing an element of $\girls$
determined by $\phi$; $2$ indicates that you are shadowing an random
element of $\girls$; and $1,2$ indicates that you are shadowing either
of these two.  The tildes indicate the partition element that you are
aiming for (i.e.\  the partition element that $\psi(x,z)$ lies in)
rather than the partition element that $y$ actually ends up lying in.
Also the superscript `BL' indicates that you are getting a whole
block's worth of information at once, whereas otherwise you get the
information a symbol at a time.  The partition $\mathcal
Q_{\ell}^{1,2}$ tells you which element of $\mathcal Q_{\ell}$ the
point $y$ ends up in for the parts that are constrained by
$\psi(x,z)$.  The partition $\mathcal Q_{\ell}^{\Bl 1,2}$ tells you
which element of $\bigvee_{i=M}^{N-10M-1}S^{-i}\mathcal Q_{\ell}$ the
point $y$ ends up in if $x$ is at the base of the Rokhlin tower and is
at the start of a $\boys$ block; $\tilde{\mathcal Q}_{\ell}^{1,2}$
tells you which element of $\bigvee_{i=M}^{N-10M-1}S^{-i}\mathcal
Q_{\ell}$ the orbit segment you are aiming for (${(\psi(x,z)_i)}_{M
  \le i<N-10M}$ ) belongs to when $x$ is at the base of the Rokhlin
tower and is at the start of a $\boys$ block.  The partition
$\tilde{\mathcal Q}_{\ell}^1$ tells you which element of
$\bigvee_{i=M}^{N-10M-1}S^{-i}\mathcal Q_{\ell}$ the orbit segment
${(\psi(x,z)_i)}_{M \le i<N-10M}$ belongs to when you are shadowing an
element of $\girls$ determined by $\phi$; and $\tilde{\mathcal
  Q_{\ell}}^2$ tells you which element of
$\bigvee_{i=M}^{N-10M-1}S^{-i}\mathcal Q_{\ell}$ the orbit segment
${(\psi(x,z)_i)}_{M \le i<N-10M}$ belongs to when you are shadowing an
random element of $\girls$.  The partitions $\mathcal P^{\Bl j}$ (for
$j=1,2$) tell you the block $x_0^{N-1}$ when $x$ is at the base of the
tower and $B_0(x)\in\boys$ if the dictionary is being used ($j=1$) or
if the word is being randomized ($j=2$).  Note that the partition
$\mathcal P^0$ tells you the symbol $x_0$ when
$B_0(x)\not\in\boys$. or $x \in E_0$.  The partition $\mathcal P^0\vee
\mathcal P^{\Bl 1}\vee\mathcal P^{\Bl 2}$ is a generating partition
for $\mu$

Let
\[
\es:= \ns{ \omega \in \Omega: x \not \in E_0, n_0(x) = 0, B_0(x) \in
  \boys, z_0 \not = \follow},
\]
(here $\es$ stands for `entropy gain').
Then by \eqref{indep} and Corollary \ref{dict} \eqref{largeBB}, and
the independence of $\mu$ and $\zeta$, we have $\iota(\es) \geq
(\fracc{1}{N})(1-\delta)(1 - 15\delta)( \epsilon/2)$, so that by
\eqref{stupid}, we have
\begin{equation}
8N\Delta\cdot \iota(\es) >3\epsilon\Delta.\label{eq:entropygap}
\end{equation}
We also note that $C_0$ is the set of points in $\Omega$ whose first
coordinate belongs to $E_0 \ \cup \  \bigcup_{i=0}^{N-1}\tau^{-i}
\Big(F\cap\bigcup_{B\not\in\boys}[B]\Big)$. We therefore calculate
\[
\iota(C_0)=\mu(E_0) +\big((1-\mu(E_0))/N \big) \cdot
N  \cdot\mu\Big(\bigcup_{B\not\in\boys}[B]\Big)<16\delta.
\]  Using
\eqref{part-cap-delta} we obtain
\begin{equation}
\label{eq:badboyent}
\iota(C_0)\log|\mathcal P|<\epsilon\Delta.
\end{equation}
The following facts will be used to complete the calculation:

\begin{enumerate}[(a)]
\item
\label{coarser}
$ h_{\iota}(\tau,\mathcal Q_\ell|\mathcal R)\ge
h_{\iota}(\tau,\mathcal Q^{1,2}_\ell|\mathcal R)$;
\item
\label{dbarr}
$h_{\iota}(\tau,\mathcal Q^{1,2}_\ell|\mathcal R)\ge
h_{\iota}(\tau,\tilde{\mathcal Q}^{1,2}_\ell|\mathcal R)$;
\item
\label{block}
$h_{\iota}(\tau,\tilde{\mathcal Q}^{1,2}_\ell|\mathcal R)\ge
h_{\iota}(\tau,\tilde{\mathcal Q}^{\Bl 1,2}_\ell|\mathcal R)$;
\item
\label{indep2}
$h_{\iota}(\tau,\tilde{\mathcal Q}^{\Bl 1,2}_\ell|\mathcal R)
= h_{\iota}(\tau,\tilde{\mathcal Q}^{\Bl 1}_\ell|\mathcal R) +
h_{\iota}(\tau,\tilde{\mathcal Q}^{\Bl  2}_\ell|\mathcal R)$;
\item
\label{noloss}
$h_{\iota}(\tau,\tilde{\mathcal Q}^{\Bl 1}_\ell|\mathcal R)=
h_{\iota}(\tau,\mathcal P^{\Bl   1}|\mathcal R)$;
\item
\label{randomgain}
$h_{\iota}(\tau ,\tilde{\mathcal Q}^{\Bl 2}_\ell | \mathcal R) =
\mu(\es)\log|\girls|$;
\item
\label{triv}
$h_{\iota}(\tau, \mathcal P | \mathcal R) \leq h_{\iota}(\tau,
\mathcal P^{\Bl  1}|
\mathcal R) + h_{\iota}(\tau, \mathcal P^{\Bl  2}| \mathcal R) + h_{\iota}(\tau,
\mathcal P^{0}| \mathcal R)$;
\item
\label{trivv}
$h_{\iota}(\tau, \mathcal P^0 | \mathcal R) \leq \iota(C_0)\log
|\mathcal P|<\epsilon\Delta$;
\item
\label{randomloss}
$h_{\iota}(\tau, \mathcal P^{\Bl 2} | \mathcal R) \le
\mu(\es)\log|\boys|$;
\item
\label{Rent}
$h_\iota(\tau,\mathcal C)\le (\delta+\tfrac 1N)\log 3\le \epsilon\Delta.$
\end{enumerate}

Notice that \eqref{coarser},
%\eqref{dbarr},
\eqref{block} and \eqref{noloss} follow from the fact that if
$\mathcal P_1$ and $\mathcal P_2$ are partitions such that
$\sigma(\mathcal \bigvee_{i \in \Z}T^{-i}\mathcal P_1)\supseteq
\sigma(\bigvee_{i \in \Z}T^{-i}\mathcal P_2)$, then $h(\tau,\mathcal
P_1)\ge h(\tau,\mathcal P_2)$. Facts \eqref{triv} and \eqref{trivv}
follow from standard entropy results, and \eqref{eq:badboyent}.

To
see that \eqref{dbarr} holds, notice that observing the elements of
$\mathcal R$ and $\mathcal Q_\ell^{1,2}$ the point $\omega=(x,z,y)$ lies in is
sufficient to determine which element of $\girls$ was being targeted.
   Set    
\[
g(y):=\argmin_{y' \in \girls}\metric_M^{N-10M}(y,y') \text{ and }
\]
\[
G(y,k):= {S^{-n_k(x)}}g(S^{n_k(x)}(y)).
\]
  One can verify with definition of $\psi$ and the fact that  $\girls$ is $(\metric_M^{N-10M}, \eta)$-separated  that for all $k \in \Z$ we have $G(y,k) = \psi(x,z)_k$  on $C_1 \cup C_2$; furthermore, if $y' \in \q(y)$, then $G(y',k)=G(y,k)$.  (Recall that $\q:=\bigvee_{i\in \Z}{S^{-i}\mathcal  Q_\ell}$.)  Thus it suffices to show that $n_k$ restricted to $C_1 \cup C_2$ is $\mathcal R$ measurable; this follows from the fact that for $\omega \in C_1 \cup C_2$,
if $a:=\sup\ns{n <0 : \tau^n \omega \not \in C_1 \cup C_2}$ and
$b:=    \inf\ns{n >0 : \tau^n \omega \not \in C_1 \cup C_2}$, then
$b-a-1$ is finite and  a multiple of $N$.

  For \eqref{randomgain} and
\eqref{randomloss}, we work with the information functions
$I_\iota(\tilde{\mathcal Q}_\ell^{\Bl 2}|\mathcal R)$ and
$I_\iota(\mathcal P^{\Bl 2}|\mathcal R)$; these are 0 if $\omega\not\in
\es$, whereas if $\omega \in \es$, then they are $\log|\girls|$
(since each girl appears independently with equal likelihood) and at
most $\log|\boys|$,  respectively.

To establish \eqref{indep2},
it suffices to show that $h_\iota(\tilde{\mathcal Q}_\ell^{\Bl
  2}|\tilde{\mathcal Q}_\ell^{\Bl 1}\vee \mathcal R)
=h_\iota(\tilde{\mathcal Q}_\ell^{\Bl 2}|\mathcal R)$. This follows
since the second coordinate of $\Omega$ is independent of the first.

Finally, to see \eqref{Rent}, we use Abramov's formula with the
induced transformation of $\tau$ to the set 
\[
A:= \ns{ \omega \in \Omega: x \in E_0\cup
F}.
\]  
Notice that between visits to $A$, the system stays entirely in a
single element of $\mathcal C$. Hence we see that
$h_\iota(\tau,\mathcal C)\le \iota(A)\log 3\le (\delta+\tfrac 1N)\log
3$; this is bounded above by $\epsilon\Delta$ using
\eqref{part-cap-delta} and Lemma \ref{choose-N} \eqref{N-entropy-bound}.

We then have the following calculation.  By Corollary \ref{dict}
\eqref{G} and \eqref{largeBB},
\begin{equation}
\label{ratio}
\log ( \fracc{|\girls|}{|\boys|}  )   \geq
N ( h_{{\pi_2^*(\xi)}}(S) - h_{\mu}(T) - 2\Delta).
\end{equation}
    Let $h = h_{\pi_2^*(\tilde\xi)}(S)$.   We have that
\begin{align*}
h  & \geq h_{\pi_2^*(\tilde\xi)}(S,\mathcal Q_\ell) \\
 &\geq  h_{\iota}(\tau,\mathcal Q_\ell|\mathcal R)
 \geq   h_{\iota}(\tau,\mathcal Q^{1,2}_\ell|\mathcal R)
\ge h_{\iota}(\tau,\tilde{\mathcal Q}^{1,2}_\ell|\mathcal R) \text{
  (by \eqref{coarser} and \eqref{dbarr})} \\
&\ge h_{\iota}(\tau,\tilde{\mathcal Q}^{\Bl 1,2}_\ell|\mathcal R) =
h_{\iota}(\tau,\tilde{\mathcal Q}^{\Bl  1}_\ell|\mathcal R) +
h_{\iota}(\tau,\tilde{\mathcal Q}^{\Bl 2}_\ell|\mathcal R)  \text{ (by
  \eqref{block},\, \eqref{indep2})} \\
& =    h_{\iota}(\tau,\mathcal P^{\Bl  1}|\mathcal R) +
\iota(\es)\log|\girls|  \text{ (by \eqref{noloss} and \eqref{randomgain})}.
\end{align*}
Hence
\begin{align*}
h & \geq h_{\iota}(\tau, \mathcal P | \mathcal R) -  h_{\iota}(\tau,
\mathcal P^{\Bl 2}|
\mathcal R) -  h_{\iota}(\tau, \mathcal P^{0}| \mathcal R) +
\iota(\es)\log|\girls|
\text{ (by \eqref{triv})} \\
&\geq h_{\iota}(\tau,\mathcal P) -h_\iota(\tau,\mathcal C)- \epsilon\Delta +
\iota(\es)\log(|\girls|/|\boys|)
\text{ (by \eqref{trivv} and \eqref{randomloss})}\\
&\geq  h_{\mu}(T) -2\epsilon\Delta + 8N\Delta\iota(\es) \text{ (by
  \eqref{Rent},  \eqref{ratio} and \eqref{def-cap-delta})}\\
& \geq  h_{\mu}(T)  \text{ (by \eqref{eq:entropygap})}.
\qedhere
\end{align*}
\end{proof}

\bibliographystyle{abbrv}
\bibliography{embedding}

\end{document}